\nonstopmode
\documentclass[10pt,twocolumn,smallcolumn]{quick-article}

\newcommand{\uedge}[2]{\ensuremath{(#1,#2)}}

\def\mNode{\circ}      
\def\mLeaf{\bullet}    

\def\mEdge{\mNode-\mNode}
\def\mDEdge{\mNode\rightarrow\mNode}

\def\PGtext{PG}
\def\clsPG{\cls{\PGtext}}
\def\clsPGrl{\cls[\mLeaf]{\PGtext}}

\def\BGtext{BG}
\def\clsBG{\cls{\BGtext}}
\def\clsBGrl{\cls[\mLeaf]{\BGtext}}

\def\TCGtext{TCG}
\def\clsTCG{\cls{\TCGtext}}
\def\clsTCGrl{\cls[\mLeaf]{\TCGtext}}

\def\TTCGtext{TTCG}
\def\clsTTCG{\cls{\TTCGtext}}
\def\clsTTCGrl{\cls[\mLeaf]{\TTCGtext}}

\def\FCGtext{FCG}
\def\clsFCG{\cls{\FCGtext}}
\def\clsFCGrl{\cls[\mLeaf]{\FCGtext}}

\begin{document}


\title{Enumerations, Forbidden Subgraph Characterizations,\\%
  and the Split-Decomposition}

\author{%
  Maryam Bahrani\thanks{Department~of~Computer~Science, Princeton
    University, 35~Olden~Street, Princeton, NJ 08540, USA.
    \email{mbahrani@princeton.edu} and \email{lumbroso@cs.princeton.edu}}\and %
  J\'{e}r\'{e}mie Lumbroso\footnotemark[1]}
\date{}

\maketitle

\begin{abstract}
  Forbidden characterizations may sometimes be the most natural way to
  describe families of graphs, and yet these characterizations are usually
  very hard to exploit for enumerative purposes.
  
  By building on the work of Gioan and Paul~(2012) and Chauve et
  al.~(2014), we show a methodology by which we constrain a
  split-decomposition tree to avoid certain patterns, thereby avoiding the
  corresponding induced subgraphs in the original graph.
  
  We thus provide the grammars and full enumeration for a wide set of
  graph classes: ptolemaic, block, and variants of cactus graphs
  (2,3-cacti, 3-cacti and 4-cacti). In certain cases, no enumeration was
  known (ptolemaic, 4-cacti); in other cases, although the enumerations
  were known, an abundant potential is unlocked by the grammars we provide
  (in terms of asymptotic analysis, random generation, and parameter
  analyses, etc.).
  
  We believe this methodology here shows its potential; the natural next
  step to develop its reach would be to study split-decomposition trees
  which contain certain prime nodes. This will be the object of future
  work.
\end{abstract}


\section*{Introduction}

Many important families of graphs can be defined (sometimes exclusively)
through a \emph{forbidden graph characterization}. These characterizations
exist in several flavors:
\begin{enumerate}
\item \emph{Forbidden minors}, in which we try to avoid certain subgraphs from
  appearing after arbitrary edge contractions and vertex deletion.
\item \emph{Forbidden subgraphs}, in which we try to avoid certain subgraphs from
  appearing as subsets of the vertices and edges of a graph.
\item \emph{Forbidden induced subgraphs}, in which we try to avoid certain
  induced subgraphs from appearing (that is we pick a subset of vertices,
  and use all edges with both endpoints in that subset).
\end{enumerate}
As far as we know, while these notions are part and parcel of the work of
graph theorists, they are usually not exploited by analytic
combinatorists. For forbidden minors, there is the penetrating article of
Bousquet-Mélou and Weller~\cite{BoWe14}. For forbidden subgraphs or
forbidden induced subgraphs, we know of few papers, except because of the
simple nature of graphs~\cite{RaTh01}, or because some other, alternate
property is used instead~\cite{CaWo03}, or only asymptotics are
determined~\cite{RaTh04}.

We are concerned, in this paper, with \emph{forbidden induced subgraphs}.

\subsection*{Split-decomposition and forbidden induced subgraphs.}

Chauve~\etal~\cite{ChFuLu14} observed that relatively well-known graph
decomposition, called the \emph{split-decomposition}, could be a fruitful
means to enumerate a certain class called \emph{distance-hereditary
  graphs}, of which the enumeration had until then not been known (since
the best known result was the bound from Nakano~\etal~\cite{NaUeUn09},
which stated that there are at most $\cramped{2^{\lfloor 3.59n\rfloor}}$
unlabeled distance-hereditary graphs on $n$ vertices).

In addition, the reformulated version of this split-decomposition
introduced by Paul and Gioan, with internal graph labels, considerably
improved the legibility of the split-decomposition tree.

We have discovered, and we try to showcase in this paper, that the
split-decomposition is a very convenient tool by which to find induced
subpatterns: although various connected portions of the graphs may be
broken down into far apart blocks in the split-decomposition tree, the
property that there is an \emph{alternated path} between any two vertices
that are connected in the original graph is very powerful, and as we show
in Section~\ref{sec:forbidden} of this paper, allows to deduce constraints
following the appearance of an induced pattern or subgraph.

\subsection*{Outline of paper.}

In Section~\ref{sec:preliminaries}, we introduce all the definitions and
preliminary notions that we need for this paper to be relatively
self-contained (although it is based heavily in work introduced by
Chauve~\etal~\cite{ChFuLu14}).

In Section~\ref{sec:forbidden}, we introduced a collection of bijective
lemmas, which translate several forbidden patterns (a cycle with 4
vertices, a diamond, cliques, a pendant vertex and a bridge, all
illustrated in Figure~\ref{fig:forbidden}) into constraints on the
split-decomposition tree of a graph. In each of the subsequent setions, we
show how these constraints can be used to express a formal symbolic
grammar that describes the constrained tree---and by so doing, we obtain
grammar for the associated class of graphs.

We start by studying block graphs in Section~\ref{sec:block}, because
their structure is sufficiently constrained as to yield a relatively
simple grammar. We then study ptolemaic graphs in
Section~\ref{sec:ptolemaic} (which allows us to showcase how to use the
symbolic grammar to save ``state'' information, since we have to remember
the provenance of the hierarchy of each node to determine whether it has a
center as a starting point). And we finally look at some varieties of
cactus graphs in Section~\ref{sec:cactus}.

Finally, in Section~\ref{sec:conclusion}, we conclude and introduce
possible future directions in which to continue this work.


\begin{figure*}[p!]
  \centering
  \begin{bigcenter}
    \begin{minipage}[t]{.45\linewidth}
      \centering
      \includegraphics[scale=0.35]{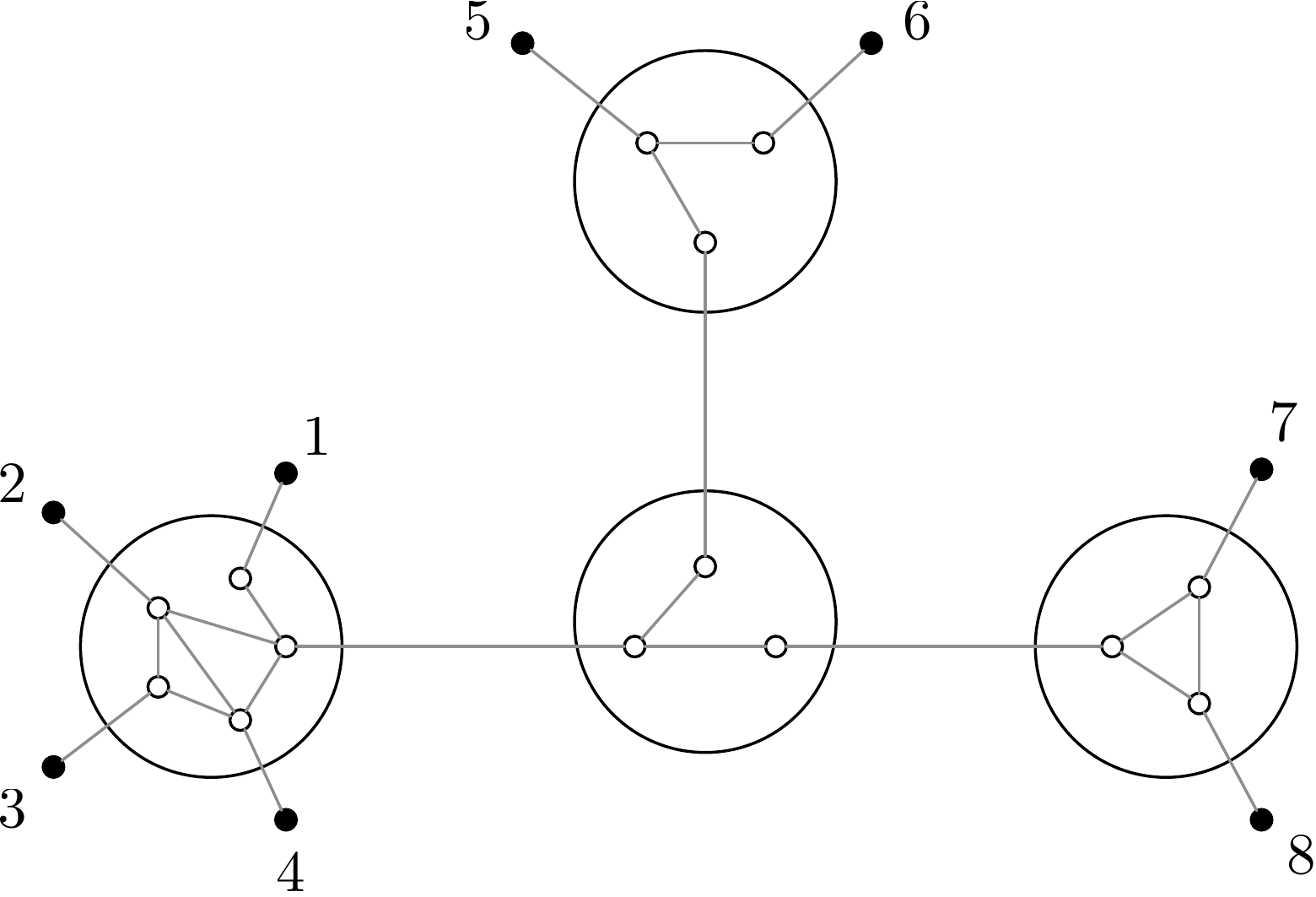}%
      \subcaption{\label{fig:ex-split-glt}%
        A graph-labeled tree.}
    \end{minipage}\hspace{2em}
    \begin{minipage}[t]{.45\linewidth}
      \centering
      \includegraphics[scale=0.35]{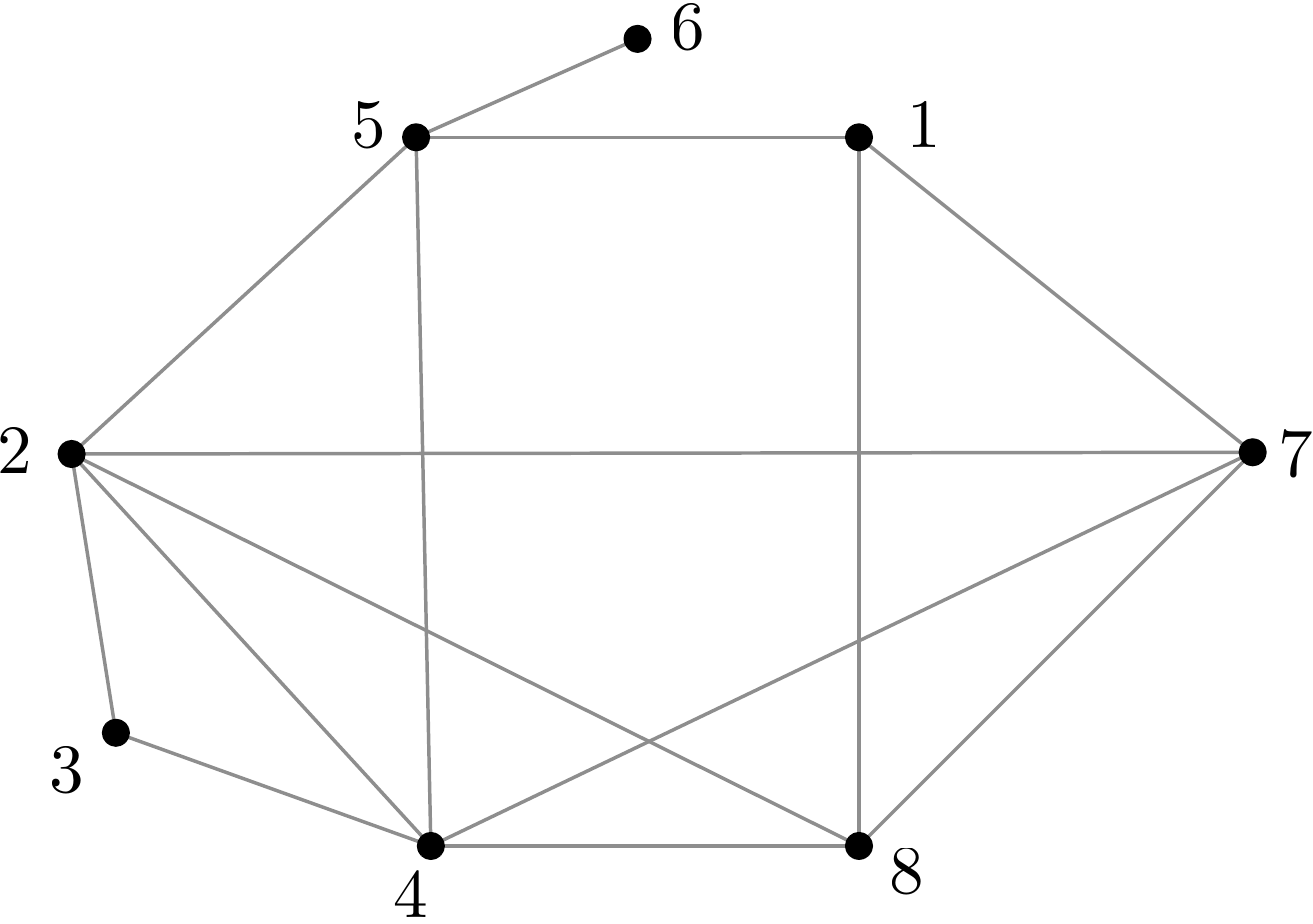}%
      \subcaption{\label{fig:ex-split-orig}%
        Original graph for (or \emph{accessibility graph} of)
        the graph-labeled tree in Figure~\ref{fig:ex-split-glt}.}
    \end{minipage}
  \end{bigcenter}
  \caption{%
    \label{fig:ex-split}%
    Two leaves of the split-decomposition graph-labeled tree (left)
    correspond to adjacent vertices in the original graph that was
    decomposed (right) if there exists an \emph{alternated path}: a path
    between those leaves, which uses at most one interior edge of any
    given graph label. For example, vertex 5 is adjacent to vertex 4 in
    the original graph, because there is an alternated path between the
    two corresponding leaves in the split-decomposition tree; vertex 5 is
    not adjacent to vertex 3 however, because that would require the path
    to take two interior edges of the (\emph{prime}) leftmost
    graph-label.}
\end{figure*}

\begin{figure*}[p!]
  \centering
  \includegraphics[scale=0.5]{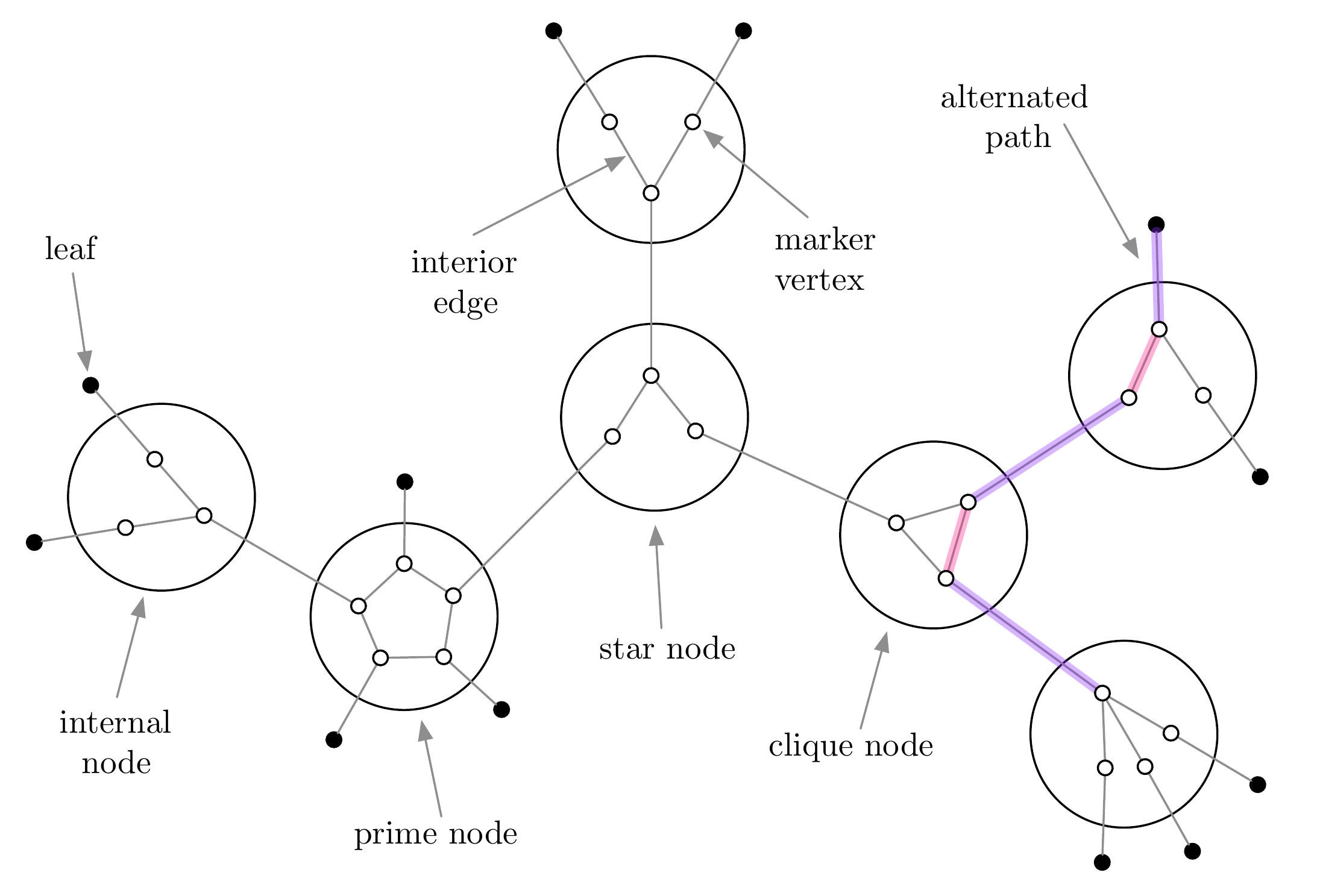}
  \caption{\label{fig:split-terminology}%
    In this figure, we present a few terms that we use a lot in this
    article.}
\end{figure*}

\begin{figure*}[p!]
  \centering
  \begin{minipage}[b]{.45\linewidth}
    \centering
    \includegraphics[scale=0.38]{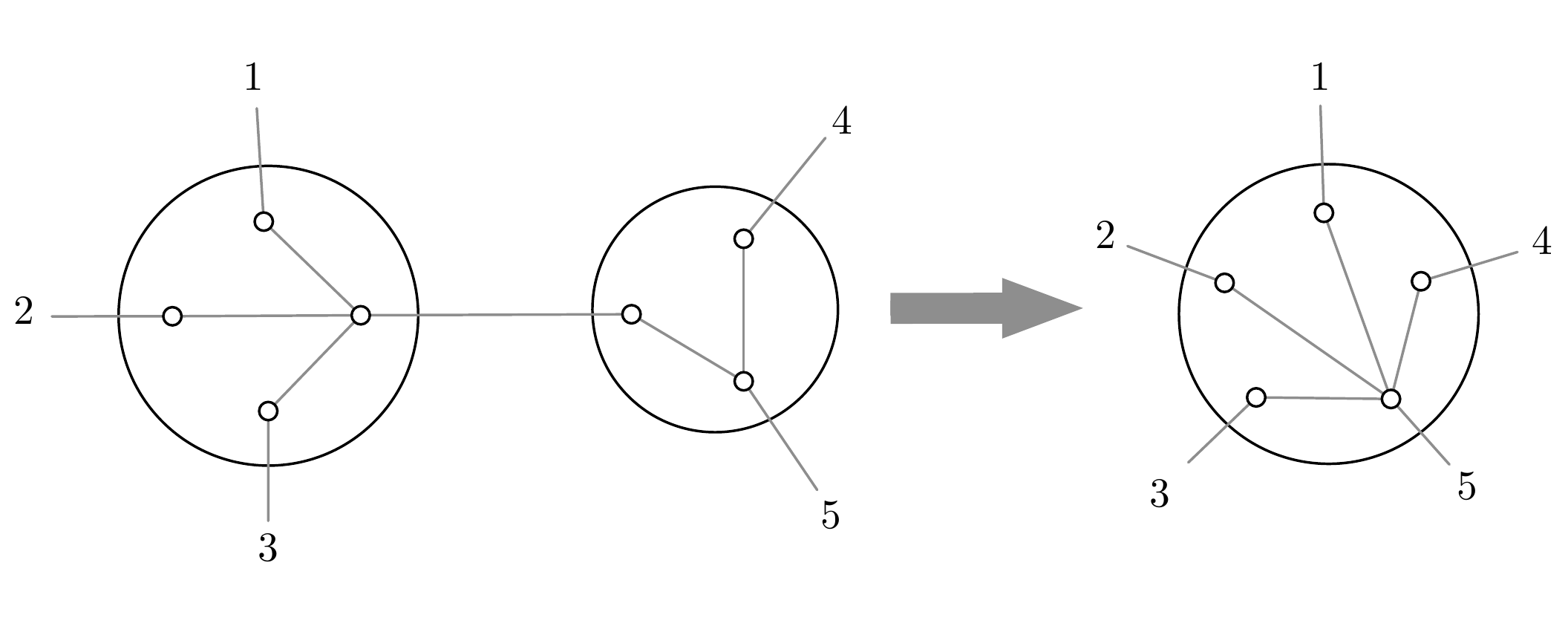}\vspace{-0.8cm}
    \subcaption{\label{fig:star-join}%
      Example of a star-join.}
  \end{minipage}\hspace{0.08\linewidth}%
  \begin{minipage}[b]{.45\linewidth}
    \centering
    \includegraphics[scale=0.38]{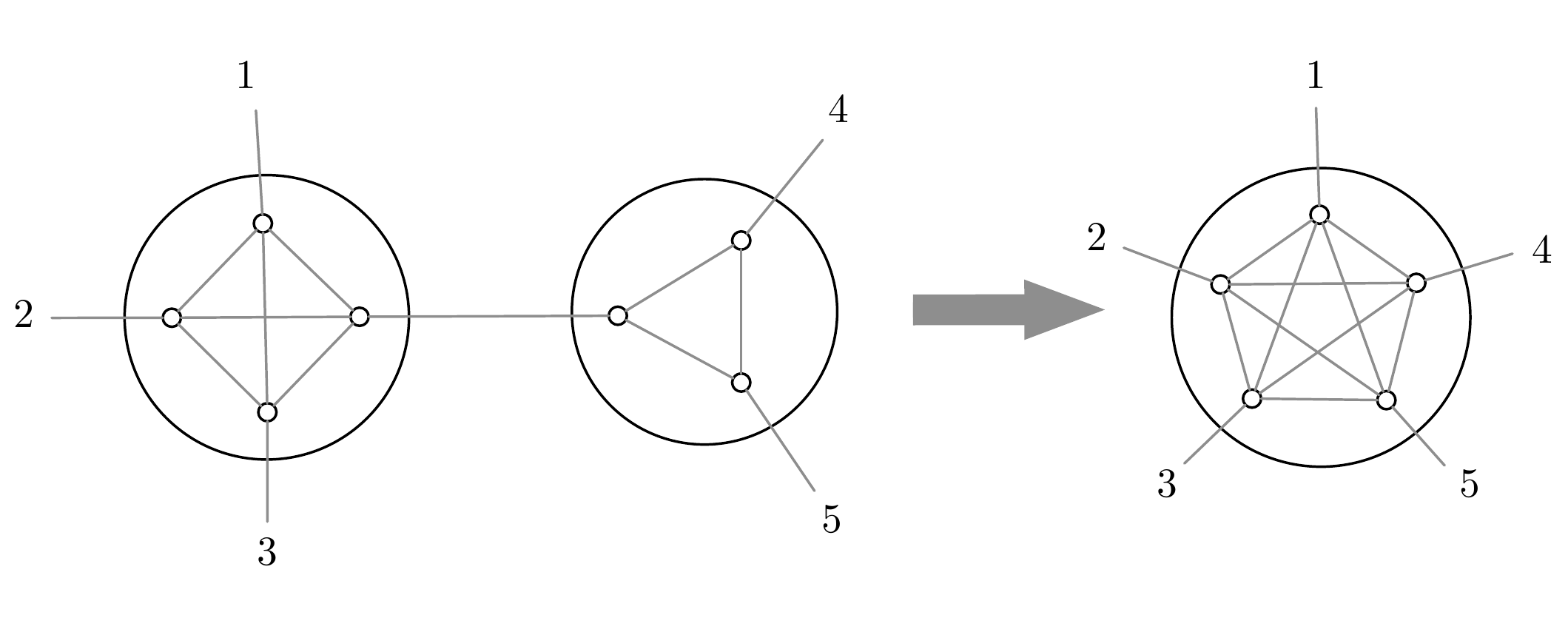}\vspace{-0.8cm}
    \subcaption{\label{fig:clique-join}%
      Example of a clique-join.}
  \end{minipage}
  \caption{\label{fig:reduced}%
    The star-join and clique-join operations result in the merging of two
    internal nodes of a split-decomposition tree. A split-decomposition
    tree in which neither one of these operations may be applied (and in
    which all non-clique and non-star nodes are prime nodes) is said to be
    \emph{reduced}.}
\end{figure*}

\section{Definitions and Preliminaries\label{sec:preliminaries}}
In this rather large section, we introduce standard definitions from graph
theory (\ref{subs:def-graph} to \ref{subs:split}) and analytic
combinatorics (\ref{subs:FS}), and then present a summary of the work of
Chauve~\etal~\cite{ChFuLu14} (\ref{subs:split-grammars}), as well as
summary of how they used the dissymmetry theorem, introduced by
Bergeron~\etal~\cite{BeLaLe98} (\ref{subs:dissymmetry}).

\subsection{Graph definitions.\label{subs:def-graph}}

For a graph $G$, we denote by $V(G)$ its vertex set and $E(G)$ its edge
set. Moreover, for a vertex $x$ of a graph $G$, we denote by $N(x)$ the
neighbourhood of $x$, that is the set of vertices $y$ such that
$\{x,y\}\in E(G)$; this notion extends naturally to vertex sets: if
$\cramped{V_1}\subseteq V(G)$, then $N(\cramped{V_1})$ is the set of
vertices defined by the (non-disjoint) union of the neighbourhoods of the
vertices in $\cramped{V_1}$. Finally, the subgraph of $G$ induced by a
subset $\cramped{V_1}$ of vertices is denoted by $G[\cramped{V_1}]$.

Given a graph $G$ and vertices $(u,v) \in \cramped{V(G)^2}$ in the same
connected component of $G$, the distance between $u$ and $v$ denoted by
$\cramped{d_G}(u, v)$ is defined as the length of the shortest path
between $u$ and $v$.

A graph on $n$ vertices is \emph{labeled} if its vertices are identified
with the set $\{1,\dots,$ $n\}$, with no two vertices having the same
label. A graph is \emph{unlabeled} if its vertices are indistinguishable.

A clique on $k$ vertices, denoted $\cramped{K_k}$ is the complete graph on
$k$ vertices (\textit{i.e.}, there exists an edge between every pair of
vertices). A star on $k$ vertices, denoted $\cramped{S_k}$, is the graph
with one vertex of degree $k-1$ (the \emph{center} of the star) and $k-1$
vertices of degree $1$ (the \emph{extremities} of the star).

\subsection{Special graph classes.\label{subs:graph-classes}}

The following two graph classes are important because they are supersets
of the classes we study in this paper.

\begin{definition}%
  \label{def:graph-dh}%
  A connected graph $G$ is \emph{distance-hereditary} if for every induced
  subgraph $H$ and every $(u,v) \in \cramped{V(H)^2}$,
  $\cramped{d_G}(u,v) = \cramped{d_H}(u,v)$.
\end{definition}

\begin{definition}%
  \label{def:graph-chordal}%
  A connected graph is \emph{chordal}, or \emph{triangulated}, or
  $\cramped{C_{\geqslant 4}}$-free, if every cycle of length at least 4
  has a chord.
\end{definition}

\subsection{Split-decomposition.\label{subs:split}}

We first introduce the notion of \emph{graph-labeled tree}, due to Gioan
and Paul~\cite{GiPa12}, then define the split-decomposition and finally
give the characterization of a \emph{reduced} split-decomposition tree,
described as a graph-labeled tree.

\begin{definition}\label{def:glt}
  A graph-labeled tree $(T,\cls{F})$ is a tree $T$ in which every internal
  node $v$ of degree $k$ is labeled by a graph
  $\cramped{G_v} \in \cls{F} $ on $k$ vertices, called \emph{marker
    vertices}, such that there is a bijection $\cramped{\rho_v}$ from the
  edges of $T$ incident to $v$ to the vertices of $\cramped{G_v}$.
\end{definition}

\noindent For example, in Figure~\ref{fig:ex-split} the internal nodes of
$T$ are denoted with large circles, the marker vertices are denoted with
small hollow circles, the leaves of $T$ are denoted with small solid
circles, and the bijection $\cramped{\rho_v}$ is denoted by each edge that
crosses the boundary of an internal node and ends at a marker vertex.

Importantly, the graph labels of these internal nodes are for convenience
alone---indeed the split-decomposition tree itself is unlabeled. However
as we will see in this paper, these graph-labeled trees are a powerful
tool by which to look at the structure of the original graph they
describe. Some elements of terminology have been summarized in
Figure~\ref{fig:split-terminology}, as these are frequently referenced in
the proofs of Section~\ref{sec:forbidden}.

\begin{definition}
  Let $(T, \cls{F})$ be a graph-labeled tree and let $\ell, \ell'\in V(T)$
  be leaves of $T$. We say that there is an \emph{al\-ter\-na\-ted path}
  between $\ell$ and $\ell'$, if there exists a path from $\ell$ to
  $\ell'$ in $T$ such that for any adjacent edges $e = \uedge{u,v}$ and
  $e' = \uedge{v,w}$ on the path,
  $\uedge{\cramped{\rho_v}(e),\cramped{\rho_v}(e')}\in E(\cramped{G_v})$.
\end{definition}

\begin{definition}%
  \label{def:split-originalgraph}%
  The \emph{original graph}, also called \emph{accessibility graph}, of a
  graph-labeled tree $(T, \cls{F})$ is the graph $G = G(T, \cls{F})$ where
  $V(G)$ is the leaf set of $T$ and, for $x, y\in V(G)$, $(x,y)\in E(G)$ iff
  $x$ and $y$ are accessible in $(T, \cls{F})$.
\end{definition}

\noindent Figures~\ref{fig:ex-split} and~\ref{fig:split-terminology}
illustrate the concept of alternated path: it is, more informally, a path
that only ever uses at most one interior edge of graph-label.

\begin{definition}%
  \label{def:split}%
  A \emph{split}~\cite{Cunningham82} of a graph $G$ with vertex set $V$ is
  a bipartition $(\cramped{V_1},\cramped{V_2})$ of $V$ (\textit{i.e.},
  $V=\cramped{V_1}\cup V_2$, $\cramped{V_1}\cap \cramped{V_2}=\emptyset$)
  such that
  \begin{enumerate}[label=(\alph*), noitemsep, nosep]
  \item $|V_1|\geqslant 2$ and $|V_2|\geqslant 2$;
  \item every vertex of $N(V_1)$ is adjacent to every of $N(V_2)$.
  \end{enumerate}
\end{definition}

\noindent A graph without any split is called a \emph{prime} graph. A
graph is \emph{degenerate} if any partition of its vertices without a
singleton part is a split: cliques and stars are the only such 
graphs.

Informally, the split-decomposition of a graph $G$ consists in finding a
split $(\cramped{V_1}, \cramped{V_2})$ in $G$, followed by decomposing $G$
into two graphs $\cramped{G_1}=G[\cramped{V_1}\cup \{\cramped{x_1}\}]$
where $\cramped{x_1}\in N(\cramped{V_1})$ and
$\cramped{G_2}=G[\cramped{V_2}\cup \{\cramped{x_2}\}]$ where
$\cramped{x_1}\in N(\cramped{V_2})$ and then recursively decomposing
$\cramped{G_1}$ and $\cramped{G_2}$. This decomposition naturally defines
an unrooted tree structure of which the internal vertices are labeled by
degenerate or prime graphs and whose leaves are in bijection with the
vertices of $G$, called a \emph{split-decomposition tree}. A
split-decomposition tree $(T,\mathcal{F})$ with $\mathcal{F}$ containing
only cliques with at least three vertices and stars with at least three
vertices is called a \emph{clique-star tree}\footnote{In this paper, we
  only consider split-decomposition trees which are clique-star trees. As
  such the family $\mathcal{F}$, to which our graph labels belong, is
  understood to only contain cliques and stars: we thus omit
  $\mathcal{F}$, and simply refer to clique-star trees as $T$.}

It can be shown that the split-decomposition tree of a graph might not be
unique (\textit{i.e.}, several sequences of decompositions of a given
graph can lead to different split-decomposition trees), but following
Cunningham~\cite{Cunningham82}, we obtain the following uniqueness result,
reformulated in terms of graph-labeled trees by Gioan and
Paul~\cite{GiPa12}.

\begin{theorem*}[Cunningham~\cite{Cunningham82}]%
  \label{thm:cunningham}%
  For every connected graph $G$, there exists a unique split-decomposition
  tree such that:
  \begin{enumerate}[label=(\alph*), noitemsep, nosep]
  \item every non-leaf node has degree at least three;
  \item no tree edge links two vertices with clique labels;
  \item no tree edge links the center of a star-node to the extremity of
    another star-node.
  \end{enumerate}
\end{theorem*}

\noindent Such a tree is called \emph{reduced}, and this theorem
establishes a one-to-one correspondence between graphs and their reduced
split-decomposition trees. So enumerating the split-decomposition trees of
a graph class provides an enumeration for the corresponding graph class,
and we rely on this property in the following sections.

Figure~\ref{fig:reduced} demonstrates the \emph{star-join} and
\emph{clique-join} operations which respectively allow trees that do not
verify conditions (b) and (c) to be further reduced---in terms of number
of internal nodes.

\begin{lemma}[Split-decomposition tree characterization of
  distance-hereditary graphs~\cite{Cunningham82, GiPa12}]%
  \label{lem:dh-split-characterization}
  A graph is \emph{distance-hereditary} if and only its
  split-decomposition tree is a clique-star tree. For this reason,
  distance-hereditary graphs are called \emph{totally decomposable} with
  respect to the split-decomposition.
\end{lemma}

\subsection{Decomposable structures.\label{subs:FS}}

In order to enumerate classes of split-decomposition trees, we use the
framework of decomposable structures, described by Flajolet and
Sedgewick~\cite{FlSe09}. We refer the reader to this book for details and
outline below the basic idea.

We denote by $\clsAtom$ the combinatorial family composed of a single
object of size $1$, usually called \emph{atom} (in our case, these refer
to a leaf of a split-decomposition tree, \textit{i.e.}, a vertex of the
corresponding graph).

Given two disjoint families $\cls{A}$ and $\cls{B}$ of combinatorial
objects, we denote by $\cls{A} + \cls{B}$ the \emph{disjoint union} of the
two families and by $\cls{A} \times \cls{B}$ the \emph{Cartesian product}
of the two families.

Finally, we denote by $\Set{\cls{A}}$ (resp.
$\cramped{\Set[\geqslant k]{\cls{A}}}$, $\cramped{\Set[k]{\cls{A}}}$) the
family defined as all sets (resp. sets of size at least $k$, sets of size
exactly $k$) of objects from ${\cls{A}}$, and by
$\cramped{\Seq[\geqslant k]{\cls{A}}}$, the family defined as all
sequences of at least $k$ objects from ${\cls{A}}$.

\subsection{Split-decomposition trees expressed symbolically.%
  \label{subs:split-grammars}}

While approaching graph enumeration from the perspective of tree
decomposition is not a new idea (the recursively decomposable nature of
trees makes them well suited to enumeration), Chauve~\etal~\cite{ChFuLu14}
brought specific focus to Cunningham's split-decomposition.

Their way of describing constrained split-decomposition trees with
decomposable grammars is the starting point of this paper, so we briefly
outline their method here.

\paragraph{Example} Let us consider the split-decomposition tree of
Figure~\ref{fig:ex-split-glt}, and illustrate how this
tree\footnote{Figure~\ref{fig:ex-split-glt} is not a clique-star tree
  because it contains a prime node---the leftmost internal node that does
  not have any splits. We illustrate the method for this more general
  split-decomposition tree, noting that the process would be identical in
  the case of a clique-star tree.} can be expressed recursively as a
rooted tree.

Suppose the tree is rooted at vertex 5. Assigning a root immediately defines
a direction for all tree edges, which can be thought of as oriented away from
the root. Starting from the root, we can set out to traverse the tree in the
direction of the edges, one internal node at a time.

We start at the root, vertex 5. The first internal node we encounter
is a star-node, and since we are entering it from the star's center, we
have to describe what is on each of its two remaining extremities. On one
of the extremities there is a leaf, 6; on the other, there is another
split-decomposition subtree, of which the first internal node we
encounter happens to be another star-node.

This time, we enter the star-node through one of its extremity. So we must
describe what is connected to its center and its remaining extremities (of
which there is only one).

Both of these are connected to smaller split-decomposition trees: the
extremity is connected itself to a clique-node, which we enter through one
of its undistinguished edges (leaving the two other to go to leaves, 7 and
8); the center of the star-node is connected to a prime node, and so on.

\paragraph{Grammar description} Now, to describe this tree symbolically,
let's consider the rule for star-nodes (assuming we are, unlike in the
tree of Figure~\ref{fig:ex-split-glt}, in a clique-star tree that has no
prime internal nodes). First assume like at the beginning of our example,
that we enter a star-node through its center: we have to describe what the
extremities can be connected to.

According to Cunningham's Theorem: we know that there are at least two
extremities (since every non-leaf node has degree at least three); and we
know that the star-node's extremities \emph{cannot} be connected to the
center of another star-node. We call $\cramped{\cls[C]{S}}$ a
split-decomposition tree that is traversed starting at a star-node entered
through its center. We have
\begin{align*}
  \cls[C]{S} = \Set[\geqslant 2]{\clsAtom + \cls{K} + \cls[X]{S}}
\end{align*}
because indeed, we have at least two extremities, which are not
ordered---so $\cramped{\Set[\geqslant 2]{\ldots}}$---and each of these
extremities can either lead to a leaf, $\clsAtom$, a clique-node entered
through any edge, $\cls{K}$, and a star-node \emph{entered through one of
  its extremities}, $\cramped{\cls[X]{S}}$.

For a star-node entered through its extremity, we have a similar
definition, with a twist,
\begin{align*}
  \cls[X]{S}  &= \left(\clsAtom +\cls{K}+\cls[C]{S}\right)\times
                \Set[\geqslant 1]{\clsAtom + \cls{K} + \cls[X]{S}}
\end{align*}
because the center---which can lead to a leaf, $\clsAtom$, a clique-node,
$\cls{K}$, or a star-node entered through its center,
$\cramped{\cls[C]{S}}$---is distinct from the extremities (which, from the
perspective of the star-node itself, are undistinguishable). We thus
express the subtree connected through the center as separate from those
connected through the extremities: this is the reason for the Cartesian
product (rather than strictly using non-ordered constructions such as
$\Set$).

\paragraph{Conventions} As explained above, we use rather similar
notations to describe the combinatorial classes that arise from
decomposing split-decomposition trees. These notations are summarized in
Table~\ref{tab:symbols}, and the most frequently used are:
\begin{itemize}[noitemsep]
\item $\cls{K}$ is a clique-node entered through one of its edges;
\item $\cramped{\cls[C]{S}}$ is a star-node entered through its center;
\item $\cramped{\cls[X]{S}}$ is a star-node entered through one of its
  extremities.
\end{itemize}

\noindent Furthermore because we provide grammars for tree classes that
are both rooted and unrooted, we use some notation for clarity. In
particular, we use $\cramped{\clsAtom_{\bullet}}$ to denote the
\emph{rooted vertex}, although this object does not differ in any way from
any other atom $\clsAtom$.

\paragraph{Terminology} In the rest of this paper, we describe the
combinatorial class $\cramped{\cls[X]{S}}$ as representing a ``\emph{a
  star-node entered through an extremity}'', but others may have alternate
descriptions: such as ``\emph{a star-node linked to its parent by an
  extremity}''; or such as Iriza~\cite{Iriza15}, ``\emph{a star-node with
  the subtree incident to one of its extremities having been
  removed}''---all these descriptions are equivalent (but follow different
viewpoints).

\subsection{The dissymmetry theorem.\label{subs:dissymmetry}}

All the grammars produced by this methodology are \emph{rooted} grammars:
the trees are described as starting at a root, and branching out to
leaves---yet the split-decomposition trees are not rooted, since they
decompose graphs which are themselves not rooted.

If we were limiting ourselves to \emph{labeled} objects\footnote{Labeled
  objects are composed of atoms (think of atoms as being vertices in a
  graph, or leaves in a tree) that are each uniquely distinguished by an
  integer between 1 and $n$, the size of the object; each of these integer
  is called a \emph{label}.}, it would be simple to move from a rooted
object to an unrooted one, because there are exactly $n$ ways to root a
tree with $n$ labeled leaves. But because we allow the graphs (and
associated split-decomposition trees) to be \emph{unlabeled}, some
symmetries make the transition to unrooted objects less straightforward.

While this problem has received considerable attention since
P\'olya~\cite{Polya37, PoRe87}, Otter~\cite{Otter48} and
others~\cite{HaUh53}, we choose to follow the lead of
Chauve~\etal~\cite{ChFuLu14}, and appeal to a more recent result, the
\emph{dissymmetry theorem}. This theorem was introduced by
Bergeron~\etal~\cite{BeLaLe98} in terms of ordered and unordered pairs of
trees, and was eventually reformulated in a more elegant manner, for
instance by Flajolet and Sedgewick~\cite[VII.26 p.~481]{FlSe09} or
Chapuy~\etal~\cite[\S 3]{ChFuKaSh08}. It states
\begin{align}\label{eq:dissymmetry}
  \cls{A} + {\cls{A}}_{\mDEdge} \simeq
  {\cls{A}}_{\mNode} + {\cls{A}}_{\mEdge}
\end{align}
where $\cls{A}$ is the unrooted class of trees, and
$\cramped{{\cls{A}}_{\mNode}}$, $\cramped{{\cls{A}}_{\mEdge}}$,
$\cramped{{\cls{A}}_{\mDEdge}}$ are the rooted classes of trees
respectively where only the root is distinguished, an edge from the root
is distinguished, and a directed, outgoing edge from the root is
distinguished. The proof is straightforward, see Drmota~\cite[\S 4.3.3,
p.~293]{HoECDrmota15}, and involves the notion of \emph{center} of a tree.


\begin{figure*}
  \centering
  \begin{minipage}[t]{.25\linewidth}
    \centering
    \includegraphics[scale=0.5]{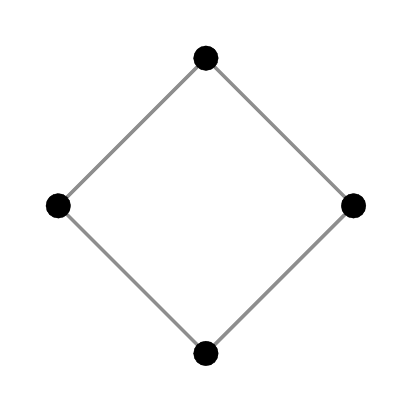}
    \subcaption{A $\cramped{C_4}$ (cycle with 4 vertices).}
  \end{minipage}
  \begin{minipage}[t]{.25\linewidth}
    \centering
    \includegraphics[scale=0.5]{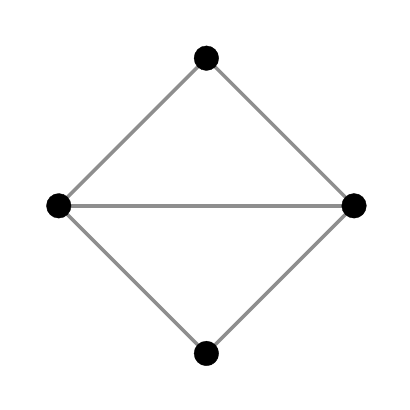}
    \subcaption{A diamond.}
  \end{minipage}\vspace{1em}\\
  \begin{minipage}[b]{.25\linewidth}
    \centering
    \includegraphics[scale=0.5]{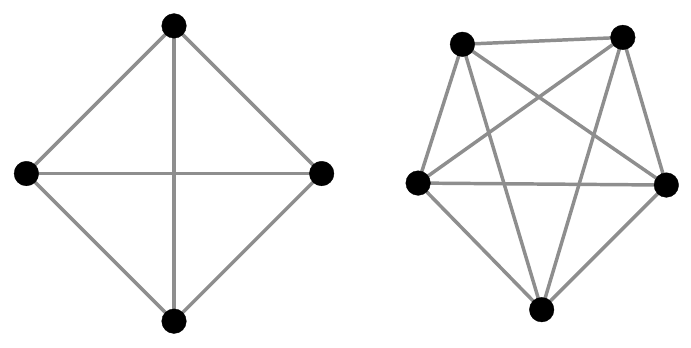}
    \subcaption{Cliques (here, of size 4 and 5).}
  \end{minipage}
  \begin{minipage}[b]{.25\linewidth}
    \centering
    \includegraphics[scale=0.5]{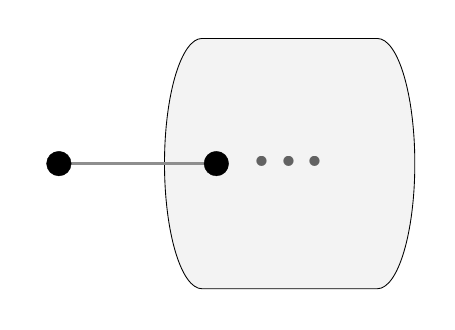}%
    \subcaption{A pendant vertex.}
  \end{minipage}
  \begin{minipage}[b]{.25\linewidth}
    \centering
    \includegraphics[scale=0.5]{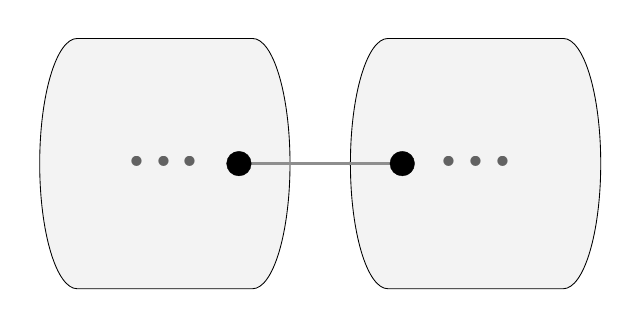}%
    \subcaption{A bridge.}
  \end{minipage}
  \caption{%
    \label{fig:forbidden}%
    These are the induced forbidden subgraphs that we investigate in this
    paper. In Section~\ref{sec:forbidden}, we introduce a series of lemmas
    that characterize the split-decomposition tree of a (totally
    decomposable) graph which avoids one or some of these induced
    subgraphs.}
\end{figure*}

For more details on the dissymmetry theorem, see Chauve~\etal~\cite[\S 2.2
and \S 3]{ChFuLu14}. We will content ourselves with some summary remarks:
\begin{itemize}
\item The process of applying the dissymmetry theorem involves
  \emph{rerooting} the trees described by a grammar in every possible way.
  Indeed, the trees obtained from our methodology will initially be rooted
  at their \emph{leaves}. For the dissymmetry theorem, we re-express the
  grammar of the tree in all possible ways it can be rooted.
\item A particularity of the dissymmetry theorem is that in this rerooting
  process, we can completely ignore leaves~\cite[Lemma~1]{ChFuLu14}, as
  the effect of doing this cancels out in the subtraction of
  Eq.~\eqref{eq:dissymmetry}:

  \begin{lemma}[Dissymmetry theorem leaf-invariance~\cite{ChFuLu14}]%
    \label{lem:no-leaves}%
    In the dissymmetry theorem for trees, when rerooting at the nodes (or
    atoms) of a combinatorial tree-like class $\cls{A}$, leaves can be
    ignored.
  \end{lemma}

\item In terms of notation, we systematically refer to $\cls[\omega]{T}$
  as trees re-rooted in a node (or edge) of type $\omega$. Often these
  rerooted trees present the distinct characteristic that, unlike the
  trees described in the rooted grammars, they are not ``missing a
  subtree.'' Thus the combinatorial class $\cramped{\cls[S]{T}}$ refers to
  a split-decomposition tree (of some graph family) that is rerooted at
  star-nodes: in this context, we must account both for the center, and at
  least two extremities.

\item This is a relatively simple theorem to apply; the downside is that
  it only yields an equality of the coefficient, but it loses the symbolic
  meaning of a grammar. This is a problem when using the tools of analytic
  combinatorics~\cite{FlSe09}, in particular those having to do with
  random generation~\cite{FlZiVa94, DuFlLoSc04, FlFuPi07}.
\item An alternate tool to unroot combinatorial classes,
  \emph{cycle-pointing}~\cite{BoFuKaVi11}, does not have this issue: it is
  a combinatorial operation (rather than algebraic one), and it allows for
  the creation of random samplers for a class. However it is more complex
  to use, though Iriza~\cite{Iriza15} has already applied it to the
  distance-hereditary and 3-leaf power grammars of
  Chauve~\etal~\cite{ChFuLu14}.
\end{itemize}


\section{Characterization \& Forbidden Subgraphs\label{sec:forbidden}}

In this section, we provide a set of bijective lemmas that characterize
the split-decomposition tree of a graph that avoids any of the forbidden
induced subgraphs of Figure~\ref{fig:forbidden}.

\subsection{Elementary lemmas.}

We first provide three simple lemmas, which essentially have to do with
the fact that the split-decomposition tree is a \emph{tree}. Their proofs
are provided in Appendix~\ref{app:proof-tree-lemmas}, and notably are
still valid in the presence of prime nodes (\textit{i.e.}, these
elementary lemmas would still apply to a split-decomposition tree that
while reduced, is not purely a clique-star tree---even though those are
the only trees that we work with in the context of this paper).

\begin{lemma}\label{lem:alternated-paths}%
  Let $G$ be a totally decomposable graph with the reduced clique-star
  split-decomposition tree $T$, any maximal\footnote{A maximal alternated path is one
    that cannot be extended to include more edges while remaining
    alternated.} alternated path starting from any node in $V(T)$ ends in
  a leaf.
\end{lemma} 

\begin{lemma}\label{lem:alternated-paths-disjoint}%
  Let $G$ be a totally decomposable graph with the reduced clique-star
  split-decomposition tree $T$ and let $u\in V(T)$ be an internal node. Any two maximal
  alternated paths $P$ and $Q$ that start at distinct marker vertices of
  $u$ but contain no interior edges from $\cramped{G_u}$ end at distinct
  leaves.
\end{lemma}

\begin{lemma}\label{lem:split-induced-clique}%
  Let $G$ be a totally decomposable graph with the reduced clique-star
  split-decomposition tree $T$. If $T$ has a clique-node of degree $n$, then $G$ has a
  corresponding induced clique on (at least) $n$ vertices.
\end{lemma}

\subsection{Forbidden subgraphs lemmas}


\begin{definition}\label{def:center-center}%
  Let $G$ be a totally decomposable graph with the reduced clique-star
  split-decomposition tree $T$. A \emph{center-center path} in $T$ is an alternated path
  $P$, such that the endpoints of $P$ are centers of star-nodes
  $(u, v) \in \cramped{V(T)^2}$ and $P$ does not contain any interior edge
  of either star-node.
\end{definition}

\begin{lemma}[Split-decomposition tree characterization of $\cramped{C_4}$-free graphs] %
  \label{lem:split-characterization-c4} %
  Let $G$ be a totally decomposable graph with the reduced clique-star
  split-decomposition tree $T$. $G$ does not have any induced $\cramped{C_4}$ if and
  only if $T$ does not have any center-center paths.
\end{lemma}

\begin{proof}
  \noindent {[$\Rightarrow$]}~~Let $T$ be a clique-star tree with a center-center
  path $P$ between the centers of two star-nodes $u,v\in V(T)$; we will
  show that the accessibility graph $G(T)$ has an induced $\cramped{C_4}$.
    
  Let $\cramped{c_u}\in\cramped{G_u}$ and $\cramped{c_v}\in\cramped{G_v}$
  be the endpoints of $P$. Since $T$ is assumed to be a reduced split-decomposition tree
  (Theorem~\ref{thm:cunningham}), $u$ and $v$ have degree at least three
  and thus $\cramped{G_u}$ and $\cramped{G_u}$ have at least two
  extremities. Therefore, there are at least two maximal alternated paths
  out of $u$ (resp. $v$), each beginning at an extremity of
  $\cramped{G_u}$ (resp. $\cramped{G_v}$) and not using any interior edges
  of $\cramped{G_u}$ (resp. $\cramped{G_v}$). By
  Lemma~\ref{lem:alternated-paths-disjoint}, these paths end at distinct
  leaves $a,b\in V(T)$ (resp. $c,d\in V(T)$), as shown in
  Figure~\ref{fig:split-subpattern-C4}.
    
  Now consider the accessibility graph $G$ of $T$. First, we observe that
  the pairs\footnote{Out of what is, perhaps, notational abuse, we refer
    to both vertices of the accessibility graph, and leaves of the
    split-decomposition tree as the same objects.} $(a,c)$, $(a,d)$,
  $(b,c)$, $(b,d)$ all belong to the edge set of $G$. We will show this
  for the edge $(a,c)$ by extending $P$ into an alternated path in $T$
  from $a$ to $c$. The argument extends symmetrically to the other three
  edges.
    
  Let $\cramped{P_a}$ be the alternated path between $a$ and an extremity
  of $\cramped{G_u}$, and let $\cramped{P_c}$ be the alternated path
  between $c$ and an extremity of $\cramped{G_v}$. To show
  $(a,c)\in E(G)$, we extend $P$ into the following alternated path:
  \begin{align*}
    P_a, c_u, P, c_v, P_c
  \end{align*}
    
  We next observe that $(a,b)$ and symmetrically $(c,d)$ cannot belong to
  the edge set of G. Since $T$ is a tree, there is a unique path in $T$
  between $a$ and $b$, which passes through $u$. This unique path must use
  two interior edges within $\cramped{G_u}$ and therefore cannot be
  alternated. Consequently, $(a,b)\not\in E(G)$. It can be shown by a
  similar argument that $(c,d)\not\in E(G)$. Therefore, the induced
  subgraph of $G$ consisting of $a,b,c,d$ is a $\cramped{C_4}$ illustrated
  in Figure~\ref{fig:split-subpattern-C4}.\medskip
      
  \begin{figure*}
    \centering
    \includegraphics[scale=0.5]{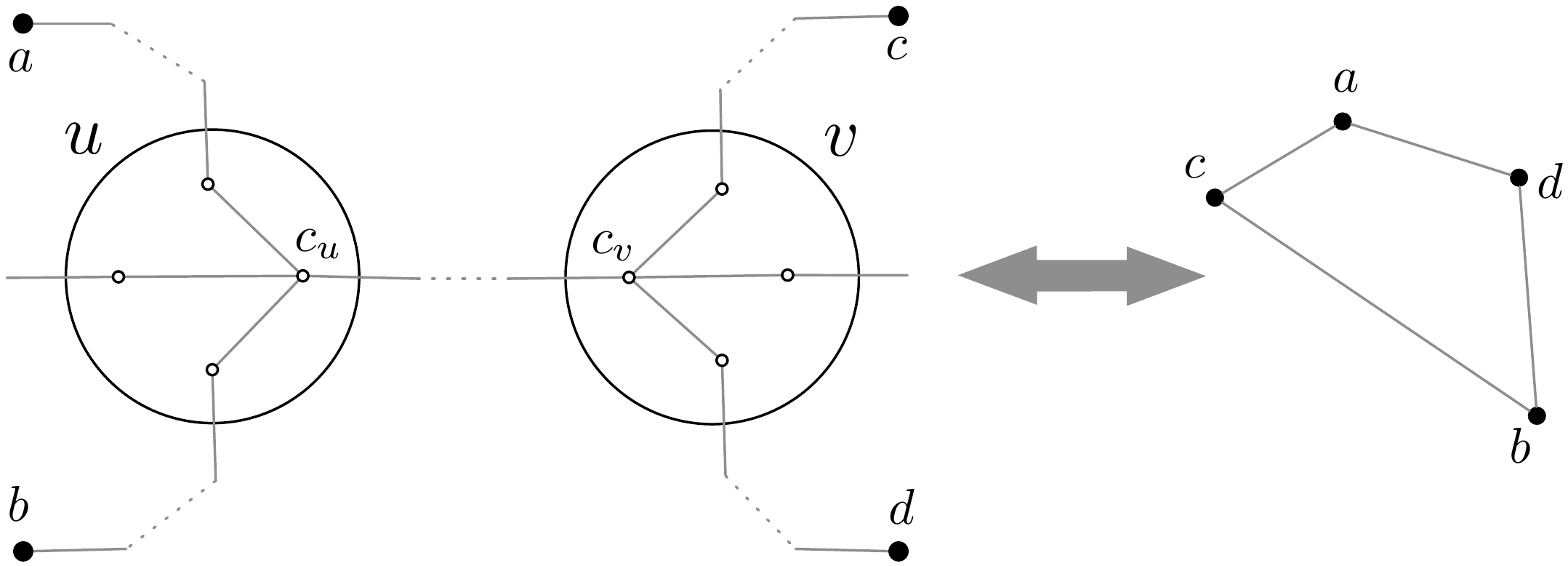}
    \caption{\label{fig:split-subpattern-C4}%
      A center-center path in a split-decomposition tree translates to an induced
      $\cramped{C_4}$ in the corresponding accessibility graph, see
      Lemma~\ref{lem:split-characterization-c4}.}
  \end{figure*}
  
  
  \noindent {[$\Leftarrow$]}~~Let $G$ be a totally decomposable graph with
  an induced $\cramped{C_4}$ with its vertices arbitrarily labeled
  $(a,c,b,d)\in V(G)$ as in Figure~\ref{fig:split-subpattern-C4}. We will
  show that the reduced split-decomposition tree $T$ of $G$ has a center-center path.
     
  First, we will show that there is a star-node $v\in V(T)$ that has
  alternated paths out of its extremities ending in $c$ and $d$. Since
  $(a,c),(a,d)\in E(G)$, there must exist alternated paths
  $\cramped{P_{a,c}}$ and $\cramped{P_{a,d}}$, which begin at the leaf $a$
  and end at the leaf $c$ or $d$ respectively. Let $v\in V(T)$ be the
  internal node that both $\cramped{P_{a,c}}$ and $\cramped{P_{a,d}}$
  enter via the same edge $\cramped{\rho_{c_v}}$ but exit via different
  edges $\cramped{\rho_{x_c}}$ and $\cramped{\rho_{x_d}}$ respectively. We
  claim that $v$ must be a star-node, such that $\cramped{c_v}$ is its
  center and $\cramped{x_c}$ and $\cramped{x_d}$ are two of its
  extremities. It is sufficient to show
  $(\cramped{x_c},\cramped{x_d})\not\in E(\cramped{G_v})$, which is indeed
  true because otherwise, we could use that edge and the disjoint parts of
  $\cramped{P_{a,c}}$ and $\cramped{P_{a,d}}$ to construct the alternated
  path between $c$ and $d$, contradicting the fact that
  $(c,d)\not\in E(G)$.
     
  Next, we will show that there is a star-node $u\in V(T)$ that has
  alternated paths out of its extremities ending in $a$ and $b$ and forms
  a center-center path with $v$. Consider this time the alternated path
  $\cramped{P_{b,c}}$ between leaves $b$ and $c$, as well as
  $\cramped{P_{a,c}}$ defined above. Similar to the argument above, let
  $u\in V(T)$ be the internal node that both $\cramped{P_{a,c}}$ and
  $\cramped{P_{b,c}}$ enter via the same edge $\cramped{\rho_{c_u}}$ but
  exit via different edges $\cramped{\rho_{x_a}}$ and
  $\cramped{\rho_{x_b}}$ respectively. With the same argument outlined
  above, $u$ must be a star-node, such that $\cramped{c_u}$ is its center
  and $\cramped{x_a}$ and $\cramped{x_b}$ are two of its extremities. It
  remains to show that $u$ and $v$ form a center-center path.

  Suppose $u$ and $v$ do not form a center-center path. Then $u$ must be
  on the common part of $\cramped{P_{a,c}}$ and $\cramped{P_{b,c}}$
  between $\cramped{x_c}$ and $c$. However, in this case, both $u$ and $v$
  lie on the unique path $\cramped{P_{b,d}}$ in $T$ between $b$ and $d$,
  in such a way that $\cramped{P_{b,d}}$ must use two interior edges of
  both $\cramped{G_u}$ and $\cramped{G_v}$, which is a contradiction since
  $(b,d)\in E(G)$ (see
  Figure~\ref{fig:split-subpattern-center-center-C4}).

  \begin{figure*}
    \centering
    \includegraphics[scale=0.5]{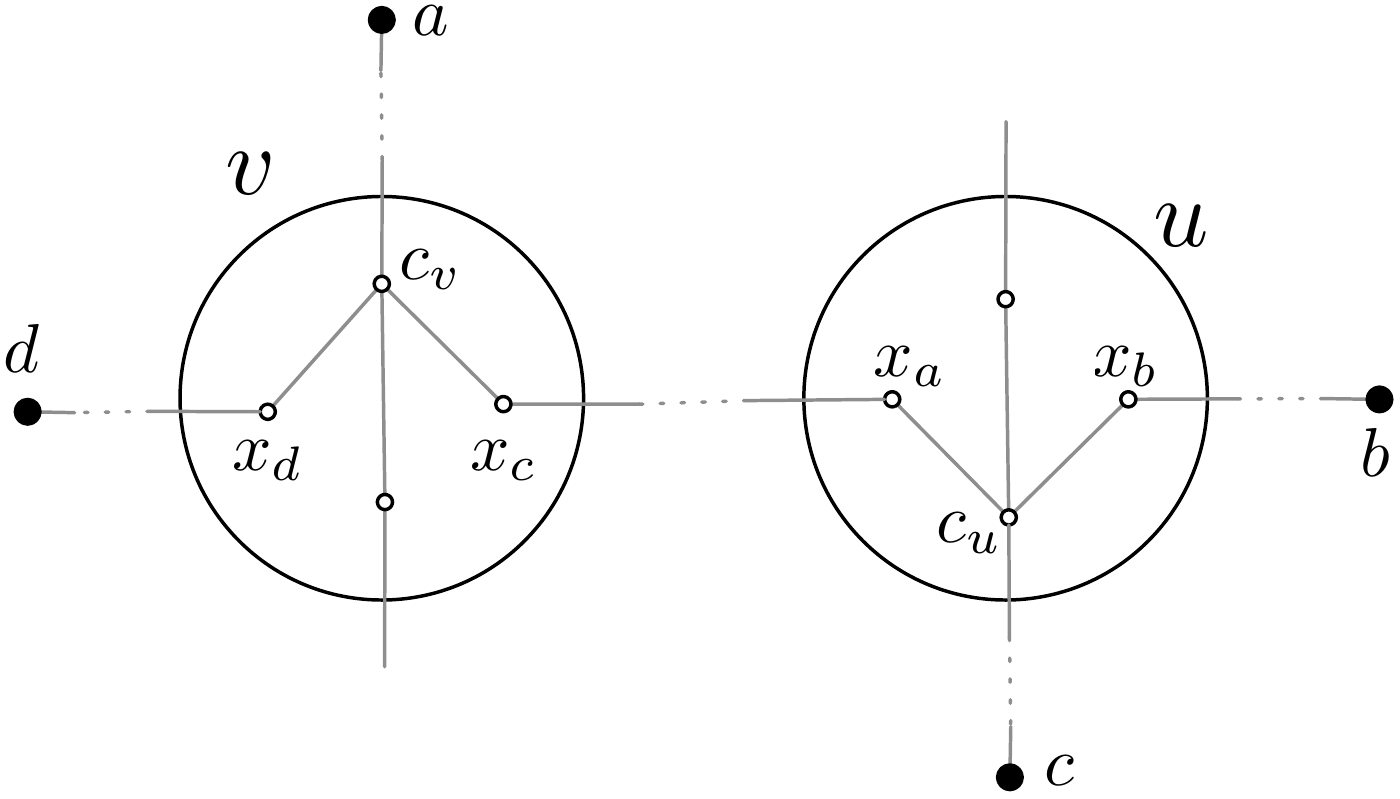}
    \caption{\label{fig:split-subpattern-center-center-C4}%
      A graph that has an induced $\cramped{C_4}$ subgraph on vertices
      $a,c,b,d$ must have a center-center path. In this figure, the star
      $u$ is where $\cramped{P_{a,c}}$ and $\cramped{P_{b,c}}$ branch
      apart, and the star-node $v$ is where $\cramped{P_{a,c}}$ and
      $\cramped{P_{a,d}}$ branch apart. If $u$ and $v$ do not form a
      center-center path, $b$ and $d$ cannot be adjacent in the
      accessibility graph $G$.}
  \end{figure*}
\end{proof}

\begin{remark}
  Importantly, a \emph{center-center path} is defined as being an
  alternated path between the centers of two star-nodes, as reflected in
  Figure~\ref{fig:split-subpattern-center-center-C4}. In this manner, the
  definition excludes the possibility that, somewhere on the path between
  the $\cramped{c_u}$ and $\cramped{c_v}$ marker vertices, there is a star
  (or for that matter a prime node) which \emph{breaks} the alternating
  path---in the sense that it requires taking at least two interior edges.
  
  But while the definition excludes it, it is a very real possibility to
  keep in mind when decomposing the grammar of the tree. As we will see in
  Section~\ref{sec:ptolemaic} on ptolemaic graphs, specifically for the
  case of the clique-node $\cls{K}$, we may need to engineer the grammar
  in such a way that it keeps track of whether a path between two nodes is
  alternated (or not).
\end{remark}


\begin{definition}%
  \label{clique-center}%
  Let $G$ be a totally decomposable graph with the reduced clique-star
  split-decomposition tree $T$. A \emph{clique-center path} in $T$ is an alternated path
  $P$, such that the endpoints of $P$ are the center of a star-node
  $u\in V(T)$ and a marker vertex of a clique-node $v\in V(T)$ and $P$
  does not contain any interior edge of the clique-node or the star-node.
\end{definition}

\begin{lemma}[split-decomposition tree characterization of diamond-free graphs]
  \label{lem:split-characterization-diamond}%
  Let $G$ be totally decomposable graph with the reduced clique-star
  split-decomposition tree $T$. $G$ does not have any induced diamonds if and only if
  $T$ does not have any induced clique-center paths.
\end{lemma}

\begin{proof}
  \noindent {[$\Rightarrow$]}~~Let $T$ be a clique-star tree containing a
  clique-center path $P$ between the center of a star-node $u\in V(T)$ and
  a marker vertex of a clique-node $v\in V(T)$. We will show that the
  accessibility graph $G(T)$ has an induced diamond.
      
  Let $\cramped{c_u}\in\cramped{G_u}$ and $\cramped{c_v}\in\cramped{G_v}$
  be the endpoints of $P$. By an argument similar to the one in the proof
  of Lemma~\ref{lem:split-characterization-c4}, it follows from
  Lemma~\ref{lem:alternated-paths-disjoint} that there must be at least
  two disjoint maximal alternated paths out of $u$, each beginning at an
  extremity of $\cramped{G_u}$ and ending at leaves $a,b\in V(T)$.
  Similarly, there must be at least two disjoint maximal alternated paths
  out of the clique-node $v$ ending at leaves $c,d\in V(T)$
  (Figure~\ref{fig:split-subpattern-diamond}).

  We can now show that this clique-center path translates to an induced
  diamond in the accessibility graph $G$ of $T$. Given this established
  labeling of the leaves $a,b,c,d$ and internal nodes $u,v$, the exact
  same argument outlined in the proof of
  Lemma~\ref{lem:split-characterization-c4} directly applies here, showing
  that $(a,c), (a,d), (b,c), (b,d) \in E(G)$. Similarly, it can be shown
  that $(a,b)\not\in E(G)$.
   
  Where this proof diverges from the proof of
  Lemma~\ref{lem:split-characterization-c4} is in the existence of the
  edge $(c,d)\in E(G)$. This is easy to show: Let
  $\cramped{x_c},\cramped{x_d}\in V(\cramped{G_v})$ be the marker vertices
  of the clique-node $v$ that mark the end points of the paths out of $v$
  to the leaves $c$ and $d$ respectively\footnote{We chose here to use the
    same notation as the proof of
    Lemma~\ref{lem:split-characterization-c4} in referring to marker
    vertices of the clique-node $v$ by names that might be reminiscent of
    the center and extremities of a star-node. This notation is not meant
    to imply that $v$ is a star-node, but rather aims to highlight the
    parallelism between the two proofs, hinting at the ease by which our
    methods can be generalized to derive split-decomposition tree characterizations for
    different classes of graphs defined in terms of forbidden subgraphs.}.
  We have $(c,d)\in E(G)$ by tracing the following alternated path:
  $c, \cramped{x_c}, \cramped{x_d}, d$ (see
  Figure~\ref{fig:split-subpattern-diamond}).\medskip
      
  \begin{figure*}
    \centering
    \includegraphics[scale=0.5]{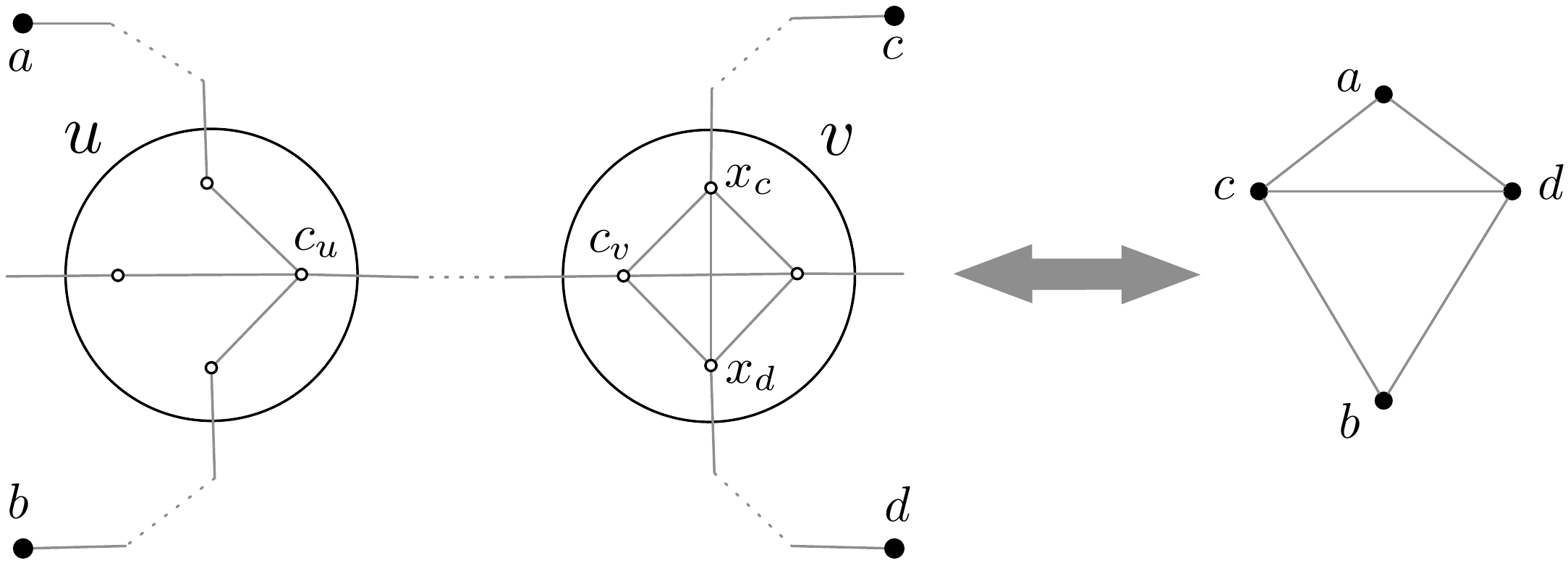}
    \caption{\label{fig:split-subpattern-diamond}%
      A clique-center path in a split-decomposition tree translates to an induced
      diamond in the accessibility graph.}
  \end{figure*}


  \noindent {[$\Leftarrow$]}~~Let $G$ be a totally decomposable graph with
  an induced diamond on vertices $(a,c,b,d)\in V(G)$ labeled as
  illustrated in Figure~\ref{fig:split-subpattern-diamond}. We need to
  show that the reduced split-decomposition tree $T$ of $G$ has a clique-center path.
     
  It can be shown by a similar argument to the proof of
  Lemma~\ref{lem:split-characterization-c4} that there must exist a star
  node $u\in V(T)$ that has alternated paths $\cramped{P_{a,c}}$ and
  $\cramped{P_{b,c}}$ out of its extremities ending in $a$ and $b$
  respectively. Let $\cramped{c_u}\in\cramped{G_u}$ be the center of this
  star-node.
     
  Similarly, we can show that there is a clique-node out of which maximal
  alternated paths lead to $c$ and $d$. Let $\cramped{P_{a,d}}$ be the
  unique path in $T$ between leaves $a$ and $d$, and consider the node
  $v\in V(T)$ where $\cramped{P_{a,c}}$ and $\cramped{P_{a,d}}$ branch
  apart. Let $\cramped{c_v}\in V(\cramped{G_v})$ denote the marker vertex
  in common between the two paths, and let
  $\cramped{x_c}, \cramped{x_d}\in V(\cramped{G_v})$ be the marker
  vertices out of which $\cramped{P_{a,c}}$ and $\cramped{P_{a,d}}$ exit
  $v$ respectively. Since $(c,d)\in E(G)$, there must be an alternated
  path in $T$ between $c$ and $d$ that uses at most one interior edge from
  $\cramped{G_v}$, so we must have
  $(\cramped{x_c}, \cramped{x_d})\in E(\cramped{G_v})$. Therefore,
  $\cramped{G_v}$ has an induced $\cramped{K_3}$ on the marker vertices
  $\cramped{x_c}$, $\cramped{x_d}$, and $\cramped{c_v}$. Since $T$ is a
  clique-star tree and $v$ cannot be a star-node, it has to be a clique
  node.
     
  Finally, we need to show that $u$ and $v$ form a clique-center path.
  This is indeed the case since, if the path $P$ between $u$ and $v$
  connected to either extremity of $\cramped{G_u}$, one of the following
  cases would occur:
  \begin{itemize}[noitemsep, nosep]
  \item $P$ connects to the extremity of $u$ ending in $a$, which implies
    $(b,c),(b,d)\not\in E(G)$;
  \item $P$ connects to the extremity of $u$ ending in $b$, which implies
    $(a,c),(a,d)\not\in E(G)$;
  \item $P$ connects to another extremity of $u$ (if one exists), which
    implies $(a,c),(a,d),(b,c),(b,d)\not\in E(G)$.
  \end{itemize}
  
  \noindent Since all the above cases contradict the fact that
  $\set{a,b,c,d}$ induces a diamond in $G$, $P$ must be a clique-center
  path between $u$ and $v$.
\end{proof}



\begin{lemma}[Split-decomposition tree characterization of graphs without induced cliques on
  4 (or more) vertices]%
  \label{lem:split-characterization-K4}%
  Let $G$ be totally decomposable graph with the reduced clique-star
  split-decomposition tree $T$. $G$ does not contain any induced
  $\cramped{K_{\geqslant4}}$ subgraphs if and only if $T$ does not have:
  \begin{itemize}[noitemsep, nosep]
  \item any clique-nodes of degree 4 or more;
  \item any alternated paths between different clique-nodes.
  \end{itemize}
\end{lemma}

\begin{proof}
  \noindent {[$\Rightarrow$]}~~We will show that for any clique-star tree
  $T$ breaking either of the conditions of this lemma, the accessibility
  graph $G(T)$ must have an induced clique on at least 4 vertices as a
  subgraph.
      
  First, suppose $T$ has a clique-node of degree 4 or more. It follows
  from Lemma~\ref{lem:split-induced-clique} that $G(T)$ must have an
  induced $\cramped{K_{\geqslant4}}$ subgraph.
      
  Second, suppose there are two clique-nodes $u,v\in V(T)$ connected via
  an alternated path $P$. Each of $\cramped{G_u}$ and $\cramped{G_v}$ must
  have at least three marker vertices, one of which belongs to $P$.
  Therefore, $u$ and $v$ each have at least two marker vertices with
  outgoing maximal alternated paths that end in two distinct leaves by
  Lemma~\ref{lem:alternated-paths-disjoint}. The four leaves at the end of
  these alternated paths are pairwise adjacent in $G$, thus inducing a
  $\cramped{K_4}$.\medskip

  
  \noindent {[$\Leftarrow$]}~~Let $G$ be a totally decomposable graph with
  an induced clique subgraph on 4 or more vertices, including
  $a,b,c,d\in V(G)$. We will show that the split-decomposition tree $T$ of $G$ breaks at
  least one the conditions listed in this lemma, i.e. either $T$ has a
  clique-node of degree 4 or more, or it has two clique-nodes (of degree
  3) connected via an alternated path.
     
  Consider the alternated paths $\cramped{P_{a,b}}$, $\cramped{P_{a,c}}$,
  and $\cramped{P_{a,d}}$ between the pairs of leaves $\{a,b\}$,
  $\{a,c\}$, and $\{a,d\}$ respectively. Let
  $\cramped{u_{b,c}}\in \cramped{P_{a,b}}\cap \cramped{P_{a,c}}$ be the
  closest internal node to $a$ in common between $\cramped{P_{a,b}}$ and
  $\cramped{P_{a,c}}$.
      
  We observe that $\cramped{u_{b,c}}$ must be a clique-node. This is the
  case because if $\cramped{u_{b,c}}$ were a star-node, at least two of
  the alternated paths would have to enter $\cramped{u_{b,c}}$ at two
  extremities and use two interior edges of the graph label
  $\cramped{g_{u_{b,c}}}$. In this case, the leaves at the end of those
  two paths could not be adjacent in $G$.
      
  By a symmetric argument, it can be shown that $\cramped{u_{b,d}}$, the
  closest internal node to $a$ in common between $\cramped{P_{a,b}}$ and
  $\cramped{P_{a,c}}$, must also be a clique-node.
      
  Depending on whether or not $\cramped{u_{b,c}}$ and $\cramped{u_{b,d}}$
  are distinct nodes, one of the conditions of the lemma is contradicted:
  \begin{itemize}[noitemsep, nosep]
  \item if $\cramped{u_{b,c}}$ and $\cramped{u_{b,d}}$ are the same clique
    node, there are four disjoint outgoing alternated paths out it,
    implying that it must have a degree of at least four, contradiction
    the first condition of the lemma;
  \item if $\cramped{u_{b,c}}$ and $\cramped{u_{b,d}}$ are distinct clique
    nodes, they are connected by an alternated path that is a part of
    $\cramped{P_{a,b}}$ between them, contradicting the second condition
    of the lemma.
  \end{itemize}
\end{proof}


\begin{lemma}[Split-decomposition tree characterization of graphs without pendant edges]%
  \label{lem:split-characterization-pendant}%
  Let $G$ be totally decomposable graph with the reduced clique-star
  split-decomposition tree $T$. $G$ does not have any pendant edges if and
  only if $T$ does not have any star-node with its center and an extremity
  adjacent to leaves.
\end{lemma}

\begin{proof}
  \noindent {[$\Rightarrow$]}~~Let $T$ be a clique-star tree, and let
  $u\in V(T)$ be a star-node, such that its center $\cramped{c_u}$ is
  adjacent to a leaf $a\in V(T)$ and one of its extremities
  $\cramped{x_b}$ is adjacent to a leaf $b\in V(T)$. We will show that $b$
  does not have any neighbors beside $a$ in the accessibility graph $G(T)$
  and thus, the edge $(a,b)$ is a pendant edge of $G(T)$.
       
  Suppose, on the contrary, that $b$ has a neighbor $c\in V(G(T))$,
  $c \neq a$. Then there must be an alternated path $P$ in $T$ that
  connects $b$ and $c$. Note that $P$ must go through $u$, entering it at
  an extremity $\cramped{x_c}\in V(\cramped{G_u})$. The path $P$ must thus
  use two interior edges $(\cramped{x_b},\cramped{c_u})$ and
  $(\cramped{c_u},\cramped{x_c})$ and cannot be alternated
  (Figure~\ref{fig:split-subpattern-pendant}).\medskip


  \begin{figure*}
    \centering
    \includegraphics[scale=0.5]{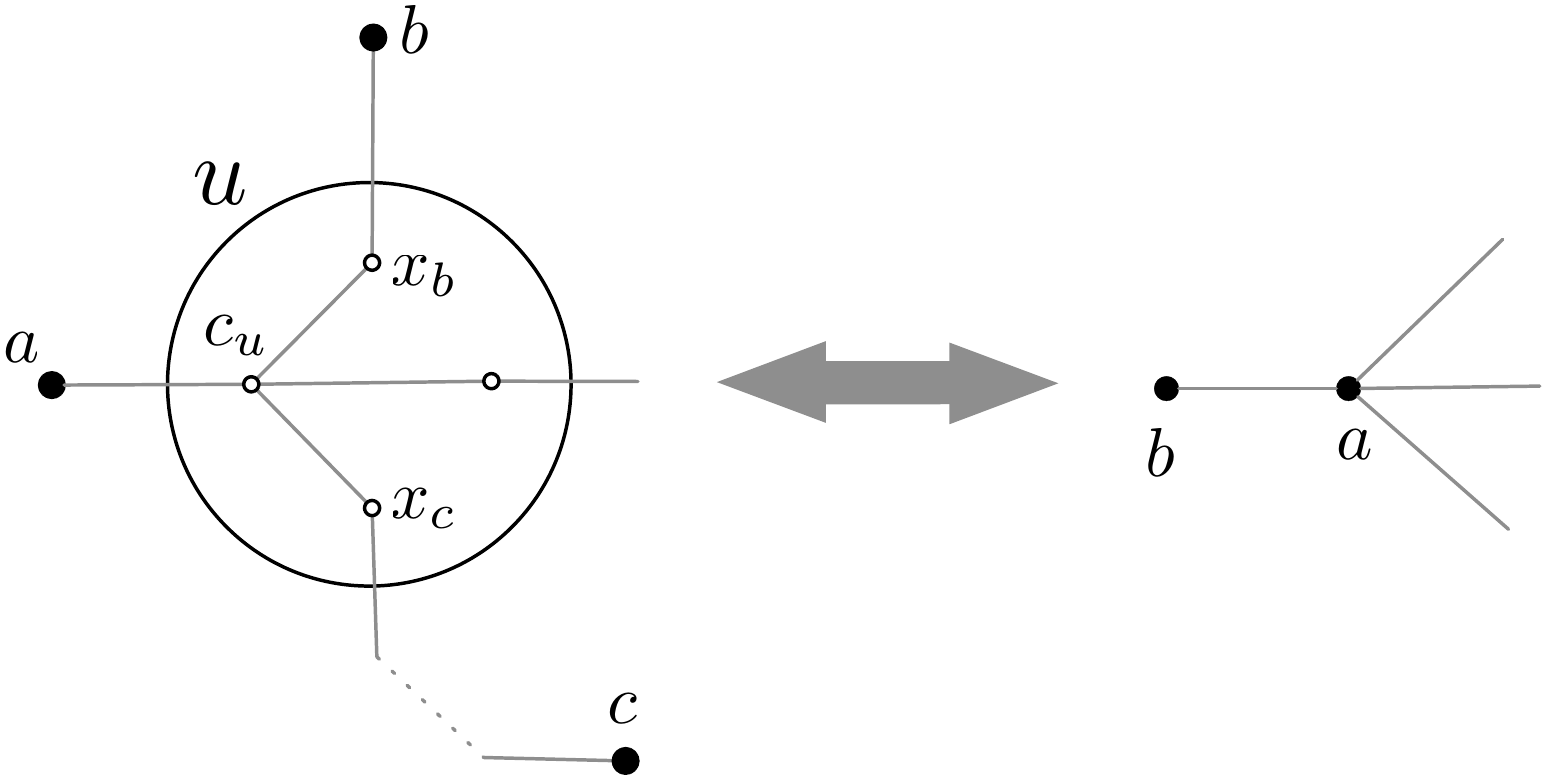}
    \caption{\label{fig:split-subpattern-pendant}%
      In a clique-star tree, a star-node with both its center and one of
      its extremities adjacent to leaves translates to a pendant edge in
      the accessibility graph.}
  \end{figure*}
  
  \noindent {[$\Leftarrow$]}~~Let $G$ be a totally decomposable graph with
  a pendant edge $(a,b)\in E(G)$ such that $b$ has degree~1
  (Figure~\ref{fig:split-subpattern-pendant}). We will show that the
  corresponding leaves $a$ and $b$ in the reduced clique-star tree $T$ of
  $G$ are attached to a star-node $u$, with its center adjacent to $a$ and
  one of its extremities adjacent to $b$.
  Let $u\in V(T)$ be the internal node to which $b$ is attached, and let
  $\cramped{x_b}\in V(\cramped{G_u})$ be the marker vertex adjacent to
  $b$.
     
  First, we will show that $u$ is indeed a star-node and $\cramped{x_b}$
  is one of its extremities. To do so, it suffices to show that
  $\cramped{x_b}$ has degree~1 in $\cramped{G_u}$. Suppose, on the
  contrary, that there are two marker vertices $y,z\in V(\cramped{G_u})$
  that are adjacent to $\cramped{x_b}$, and consider two maximal
  alternated paths out of $y$ and $z$. By
  Lemma~\ref{lem:alternated-paths-disjoint}, these paths end at two
  distinct leaves of $T$, both of which much be adjacent to $b$ in $G$,
  contradicting the assumption that $b$ has degree~1.

     
  Next, we will show that $a$ must be attached to the center
  $\cramped{c_u}$ of $\cramped{G_u}$. Otherwise, one of the following
  cases will occur:
  \begin{itemize}[noitemsep, nosep]
  \item \emph{$\cramped{c_u}$ is adjacent to a leaf $c\in V(T)$}.\ \ %
    In this case, we have the alternated path
    $b, \cramped{x_b}, \cramped{c_u}, c$, implying $(b,c)\in E(G)$,
    contradicting the assumption that $b$ has degree~1.
  \item \emph{%
      $\cramped{c_u}$ is adjacent to a clique-node $v\in V(T)$}.\ \ %
    With an argument similar to the previous case, it can be shown that in
    this case, there must exist at least two alternated paths out of $v$
    that lead to leaves, all of which must be adjacent to $b$.
  \item \emph{$\cramped{c_u}$ is adjacent to the center of a star-node
      $v\in V(T)$}.\ \ %
    Similar to the previous case, there must exist at least two alternated
    paths out extremities of $v$ that lead to leaves, all of which must be
    adjacent to $b$.
  \item \emph{$\cramped{c_u}$ is adjacent to an extremity of a star-node
      $v\in V(T)$}.\ \ %
    This case never happens, since $T$ is assumed to be a \emph{reduced}
    clique-star tree.
  \end{itemize}
  Therefore, $a$ must be attached to the center of $\cramped{G_u}$, so $u$
  is a star-node with $a$ adjacent to its center and $b$ adjacent to one
  of its extremities.
\end{proof}

\begin{lemma}[Split-decomposition tree characterization of graphs without bridges]%
  \label{lem:split-characterization-bridge}%
  Let $G$ be totally decomposable graph with the reduced clique-star
  split-decomposition tree $T$. $G$ does not have any bridges if and only
  if $T$ does not have:
  \begin{itemize}[noitemsep, nosep]
  \item any star-node with its center and an extremity adjacent to leaves;
  \item any two star-nodes adjacent via their extremities, with their
    centers adjacent to leaves.
  \end{itemize}
\end{lemma}

\begin{proof}
  We distinguish between two kinds of bridges: pendant edges and other
  bridges, which we will call \emph{internal} bridges.
  Lemma~\ref{lem:split-characterization-pendant} states that a star-node
  with its center and an extremity adjacent to leaves in $T$ corresponds
  to a pendant edge in $G$. Therefore, it suffices to show $G$ has no
  internal bridges if and only if the second condition holds in
  $T$.\medskip

  \noindent {[$\Rightarrow$]}~~Let $T$ be a clique-star tree, and let
  $u,v\in V(T)$ be two star-nodes, with the center
  $\cramped{c_u}\in\cramped{G_u}$ adjacent to a leaf $a\in V(T)$, the
  center $\cramped{c_v}\in\cramped{G_v}$ adjacent to a leaf $b\in V(T)$,
  and two of their extremities $\cramped{x_u}\in\cramped{G_u}$ and
  $\cramped{x_v}\in\cramped{G_v}$ adjacent to each other. We will show
  that $(a,b)$ is an internal bridge in $G(T)$.
            
  First, let us define the following partition of the leaves of $T$ into
  two sets: Since every edge in a tree is a bridge, removing $(u,v)$ from
  $T$ breaks $T$ into two connected components. Let
  $\cramped{V_u}, \cramped{V_v}\in V(T)$ be the leaves of these components
  respectively, and note that $a\in\cramped{V_u}$ and $b\in\cramped{V_v}$
  (Figure~\ref{fig:split-subpattern-bridge}).

  Next, note that $(a,b)\in E(G(T))$ by tracing the alternated path
  $a, \cramped{c_u}, \cramped{x_u}, \cramped{x_v}, \cramped{c_v}, b$. To
  show that $(a,b)$ must be an internal bridge, it suffices We will show
  that the edge $(a,b)$ is a bridge in $G(T)$ by showing it does not
  belong to any cycles. We will then confirm that $(a,b)$ must be an
  \emph{internal} bridge.


  Suppose, on the contrary, that there $G(T)$ has a cycle $C$ of vertices
  $(\cramped{x_1}=a,\cramped{x_2},\dots,\cramped{x_{k-1}},\cramped{x_k}=b)\in
  \cramped{V(G(T))^k}$
  for some $k\geqslant3$. Clearly, $\cramped{x_1}=a\in \cramped{V_u}$.
  Additionally, for every edge
  $(\cramped{x_i},\cramped{x_{i+1}})\in E(G(T))$, $i=1\dots k-1$, there
  must be an alternated path $\cramped{P_i}$ in $T$ between leaves
  $\cramped{x_i}$ and $\cramped{x_{i+1}}$. Furthermore, if
  $\cramped{x_i}\in\cramped{V_u}$, we must also have
  $\cramped{x_{i+1}}\in\cramped{V_u}$, since otherwise, $\cramped{P_i}$
  must use the only edge crossing the cut $\cramped{V_u}, \cramped{V_v}$;
  this requires $\cramped{P_i}$ to enter and exit $u$ via two extremities
  of $\cramped{G_u}$, which requires using two interior edges from
  $\cramped{G_u}$. Applying a similar argument for every edge
  $(\cramped{x_i},\cramped{x_{i+1}})$ of $C$ up to $b$ implies that
  $b\in\cramped{V_u}$. Therefore, we must have
  $b\in\cramped{V_u}\cap\cramped{V_v}$, contradicting the fact that
  $\cramped{V_u}$ and $\cramped{V_v}$ are disjoint.

  Finally, we can show via Lemma~\ref{lem:alternated-paths-disjoint} that
  $(a,b)$ must be an \emph{internal} bridge, by showing that $a$ and $b$
  must have neighbors besides each other in $G(T))$. We will confirm this
  for $a$, and the argument applies symmetrically to $b$. Since $T$ is
  reduced, $u$ has degree at least three, so there is at least one
  alternating path out of an extremity of $\cramped{G_u}$ other than
  $\cramped{x_u}$ ending in a leaf of $T$ other than $b$, implying that
  $a$ must be adjacent to that leaf in $G(T)$. Similarly, $b$ must have a
  neighbor in $G(T)$ other than $a$. Therefore, $(a,b)$ cannot be a
  pendant edge and must be an internal bridge.\medskip
  
  \begin{figure*}
    \centering
    \includegraphics[scale=0.5]{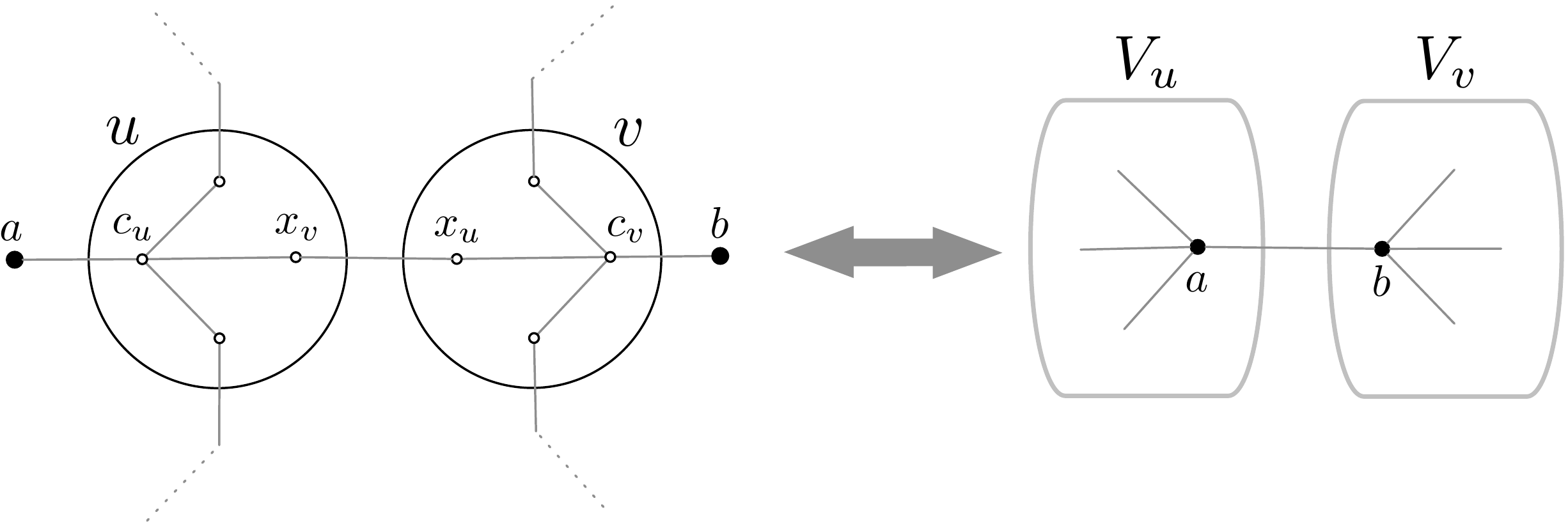}
    \caption{\label{fig:split-subpattern-bridge}%
      In a clique-star tree, a structure consisting of two star-nodes
      adjacent via their extremities, with their centers adjacent to
      leaves, translates to an \emph{internal} bridge in the accessibility
      graph.}
  \end{figure*}


  \noindent {[$\Leftarrow$]}~~Let $G$ be a totally decomposable graph with
  an internal bridge $(a,b)\in E(G)$
  (Figure~\ref{fig:split-subpattern-bridge}). We show that the
  corresponding leaves $a$ and $b$ in the reduced clique-star tree $T$ of
  $G$ are respectively attached to centers of two star-node $u$ and $v$
  adjacent via their extremities.
       
  Let $u\in V(T)$ be the internal node to which $a$ is attached, and let
  $\cramped{c_u}\in V(\cramped{G_u})$ be the marker vertex adjacent to
  $a$. Similarly, let $v\in V(T)$ be the internal node to which $b$ is
  attached, and let $\cramped{c_v}\in V(\cramped{G_v})$ be the marker
  vertex adjacent to $b$.
     
  
  Now, we show that $u$ and $v$ are star-nodes. Suppose, on the contrary,
  that $u$ is a clique-node, and note that $\cramped{G_u}$ must have at
  least two other marker vertices besides $\cramped{c_u}$. Consider the
  two maximal alternated paths $\cramped{P_x}$ and $\cramped{P_y}$ out of
  these two marker vertices, respectively ending in leaves $x,y\in V(T)$
  by Lemma~\ref{lem:alternated-paths-disjoint}. We first observe that
  $(a,x)\in E(G)$ by the union of the alternated path $\cramped{P_x}$ and
  the interior edge of $\cramped{G_u}$ between $\cramped{c_u}$ and the
  marker vertex at the end of $\cramped{P_x}$. Similarly, we have
  $(a,y)\in E(G)$. Furthermore, $(x,y)\in E(G)$ by the union of the two
  alternated paths $\cramped{P_x}$ and $\cramped{P_y}$ and the interior
  edge of $\cramped{G_u}$ between the ends of these paths.The trio of
  vertices $a,x,y\in V(G)$ thus induces a $\cramped{C_3}$ in $G$,
  contradicting the assumption that $(a,b)$ is a bridge.
       
  Next, we will show that $u$ and $v$ are adjacent to each other via their
  extremities $\cramped{x_u}$ and $\cramped{x_v}$. Otherwise, since no
  star centers are adjacent to extremities of other star-nodes in a
  reduced split-decomposition tree, $u$ and $v$ would have to be adjacent via their
  centers. This would constitute a center-center path, which would, by
  Lemma~\ref{lem:split-characterization-c4}, imply that $(a,b)$ belongs to
  a $\cramped{C_4}$ and cannot be a bridge.
       
  Finally, we confirm that $\cramped{c_u}$ and $\cramped{c_v}$, the marker
  vertices to which $a$ and $b$ are attached, are the centers of
  $\cramped{G_u}$ and $\cramped{G_v}$ respectively. It suffices to show
  this claim for $a$, as the argument symmetrically applies to $b$ as
  well. If, on the contrary, $a$ were attached to an extremity of
  $\cramped{G_u}$, the only path in $T$ between $a$ and $b$ would have to
  use two interior edges of $\cramped{G_u}$, one from $\cramped{c_u}$ to
  the center of $\cramped{G_u}$ and one from the center to
  $\cramped{x_u}$. This would imply $(a,b)\not\in E(G)$, a contradiction.
\end{proof}



\begin{figure*}
  \centering
  \includegraphics[scale=0.5]{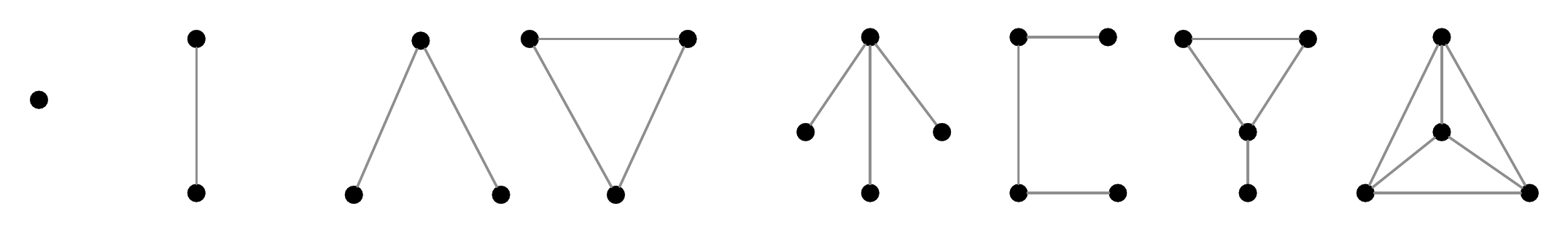}
  \caption{\label{fig:small-block-graphs}%
    All block graphs on four vertices or less.}
\end{figure*}

\begin{figure*}
\centering
\begin{bigcenter}
  \begin{minipage}[t]{.45\linewidth}
   \centering
    \includegraphics[scale=0.5]{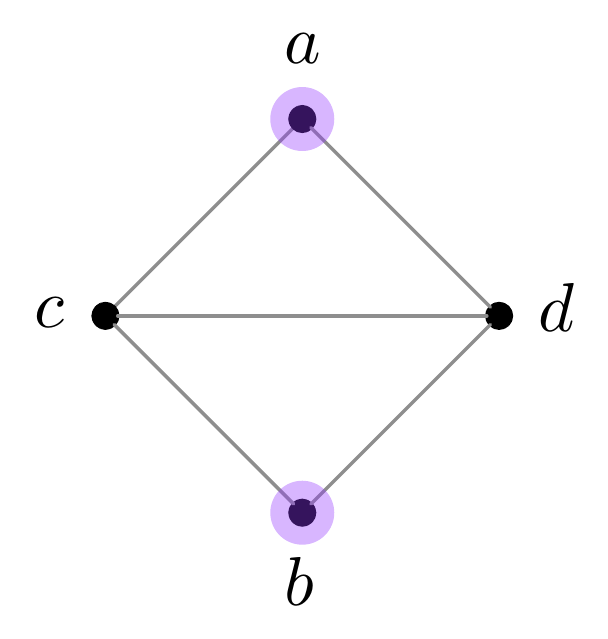}%
    \subcaption{\label{fig:weakly-geodetic-diamond}%
      A diamond.}
  \end{minipage}\hspace{2em}
  \begin{minipage}[t]{.45\linewidth}
   \centering
    \includegraphics[scale=0.5]{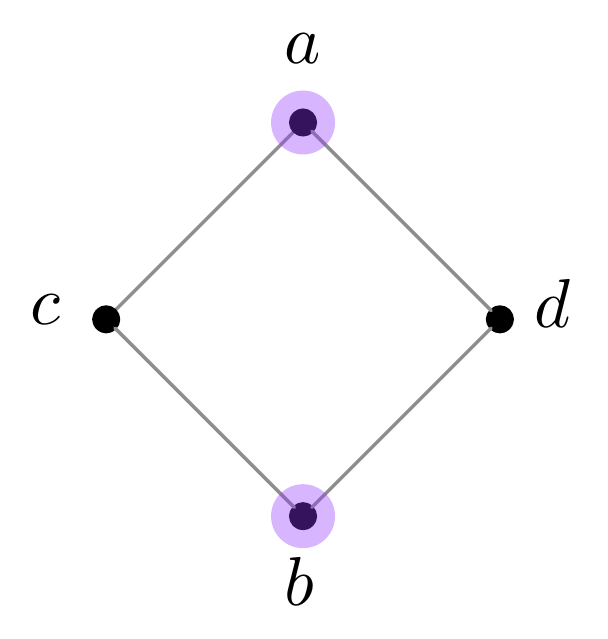}%
    \subcaption{\label{fig:weakly-geodetic-c4}%
      A $\cramped{C_4}$.}
   \end{minipage}
\end{bigcenter}
\caption{%
  \label{fig:weakly-geodetic-subgraphs}%
  The class of \emph{weakly geodetic} graphs can be characterized in terms
  of forbidden subgraphs as ($\cramped{C_4}$, diamond)-free. This is the
  case because, while in a weakly geodetic graph, a pair of vertices of
  distance 2 has a unique common neighbor, in an induced $\cramped{C_4}$
  or diamond subgraph, the highlighted vertices $a$ and $b$ are of
  distance 2 while having at least two common neighbors $c$ and $d$.}
\end{figure*}

\section{Block graphs\label{sec:block}}

In this section, we analyze a class of graphs called block graphs. After
providing a general definition of this class, we present its well-known
forbidden induced subgraph characterization, and using a lemma we proved
in Section~\ref{sec:forbidden}, we deduce a characterization of the
split-decomposition tree of graphs in this class.

Block graphs are the (weakly geodetic) subset of ptolemaic
graphs---themselves the (chordal) subset of distance-hereditary graphs.
Thus, their split-decomposition tree is a more constrained version of that
of ptolemaic graphs. As such, we use block graphs as a case study to
prepare for ptolemaic graphs, for which the grammar is a bit more
complicated.

\subsection{Characterization.}

For any graph $G$, a vertex $v$ is a \emph{cut vertex} if the number of
connected components is increased after removing $v$, and a \emph{block}
is a maximal connected subgraph without any cut vertex.

A graph is then called a \emph{block graph}~\cite{Harary63} if and only if
its blocks are complete graphs (or cliques) and the intersection of two
blocks is either empty or a cut vertex. Block graphs are the intersection
of ptolemaic graphs and weakly geodetic graphs, as was shown by Kay and
Chartrand~\cite{KaCh65}.

\begin{definition}[Kay and Chartrand~{\cite[\S 2]{KaCh65}}]%
  \label{def:weakly-geodetic}%
  A graph is \emph{weakly geodetic} if for every pair of vertices of
  distance 2 there is a unique common neighbour of them.
\end{definition}

\noindent It is relatively intuitive to figure out from this definition,
that weakly geodetic graphs are exactly ($\cramped{C_4}$, diamond)-free
graphs, but surprisingly we were only able to find this result mentioned
relatively recently~\cite{EsHoSpSr11}.

\begin{lemma}
  A graph is weakly geodetic if and only if it contains no induced
  $\cramped{C_4}$ or diamond subgraphs.
\end{lemma}

\begin{proof}
  \noindent {[$\Rightarrow$]}~~We show a weakly geodetic graph is
  ($\cramped{C_4}$, diamond)-free by arguing that graphs with induced
  $\cramped{C_4}$ subgraphs and diamonds as induced subgraphs are not
  weakly geodetic. This is illustrated in
  Figure~\ref{fig:weakly-geodetic-subgraphs}, in which the highlighted
  pairs of vertices in a $\cramped{C_4}$ and a diamond are of distance 2
  and have more than one neighbor in common.

  
  \noindent {[$\Leftarrow$]}~~Let $G$ be a ($\cramped{C_4}$,
  diamond)-free, and let $a,b\in V(G)$ be vertices of distance 2 with two
  neighbors $c,d \in V(G)$ in common. Since $a,b$ have distance 2,
  $(a,b) \not\in E(G)$. Depending on whether or not $(c,d)$ belongs to the
  edge set of $G$, we have $\{a,b,c,d\}$ inducing a diamond
  (Figure~\ref{fig:weakly-geodetic-diamond}) or a $\cramped{C_4}$
  (Figure~\ref{fig:weakly-geodetic-c4}) respectively.
\end{proof}

\noindent Since we have established that block graphs are the subset of
totally decomposable (distance-hereditary) graphs which are also
($\cramped{C_4}$, diamond)-free, we can now characterize their
split-decomposition tree by applying our two lemmas from
Section~\ref{sec:forbidden} and deducing the overall constraint on the
split-decomposition trees that these imply.

\begin{theorem}[split-decomposition tree characterization of block graphs]%
  \label{thm:split-characterization-block}%
  A graph $G$ with the reduced split-decomposition tree $(T,\mathcal{F})$ is a block
  graph if and only if
  \begin{enumerate}[label=(\alph*), noitemsep, nosep]
  \item $T$ is a clique-star tree;
  \item the centers of all star-nodes are attached to leaves.
  \end{enumerate}
\end{theorem}

\begin{proof}
  We have introduced block graphs as being the intersection class of
  ptolemaic graphs and weakly geodetic graphs. As we will see again in
  Section~\ref{sec:ptolemaic}, Howorka~\cite[\S 2]{Howorka81} has shown
  that ptolemaic graphs are the intersection class of distance-hereditary
  graphs and chordal (triangulated) graphs.

  A chordal graph is a graph in which any cycle of size larger than 3 has
  a chord; because distance-hereditary graphs are themselves
  $\cramped{C_{\geqslant 5}}$-free, chordal distance-hereditary
  (ptolemaic) graphs are the $\cramped{C_4}$-free distance-hereditary
  graphs. The additional constraint that comes with being weakly geodetic,
  implies that block graphs are the ($\cramped{C_4}$, diamond)-free
  distance-hereditary graphs\footnote{Alternatively block graphs can be
    characterized as the class of ($\cramped{C_{\geqslant4}}$,
    diamond)-free graphs. Since block graphs are also distance-hereditary,
    and since distance-hereditary graphs do not have any induced
    $\cramped{C_{\geqslant5}}$, we conclude again that block graphs can be
    thought of as ($\cramped{C_4}$, diamond)-free distance-hereditary
    graphs.}.
  
  The first condition in this theorem is due to the total decomposability
  of block graphs as a subset of distance-hereditary graphs. The second
  condition forbids having any center-center or clique-center paths,
  which, by Lemma~\ref{lem:split-characterization-c4} and
  Lemma~\ref{lem:split-characterization-diamond} respectively, ensures
  that $G$ does not have any induced $\cramped{C_4}$ or diamond.
\end{proof}

\subsection{Rooted grammar.%
  \label{subs:block-rooted-grammar}}

Using the split-decomposition tree characterization derived above, we can
provide a symbolic grammar that can be used to enumerate labeled and
unlabeled block graphs.

\begin{theorem}{\label{thm:rooted-grammar-block}}
  The class $\clsBGrl$ of block graphs rooted at a vertex is specified by
  \begin{align}
    \clsBGrl   &= \clsAtom_{\mLeaf} \times (\cls[C]{S}+\cls[X]{S}+\cls{K})\\
    \cls{K}    &= \Set[\geqslant2]{\clsAtom+\cls[X]{S}}\\
    \cls[C]{S} &= \Set[\geqslant2]{\clsAtom+\cls{K}+\cls[X]{S}}\\
    \cls[X]{S} &= \clsAtom\times\Set[\geqslant1]{\clsAtom+\cls{K}+\cls[X]{S}}\text{.}
  \end{align}
\end{theorem}

\noindent This grammar is similar to that of distance-hereditary
graphs~\cite{ChFuLu14}. The constraint that the centers of all star-nodes
are attached to leaves means essential that the rule
$\cramped{\cls[C]{S}}$ can only be reached as a starting point when we are
describing what the root vertex might be connected to (from the initial
rule, $\clsBGrl$).

For the sake of comprehensiveness, we give this proof in full detail.
However since the following proofs are fairly similar, we will tend to
abbreviate them.

\begin{proof}
  We begin with the rule for a star-node entered by its center,
  \begin{align*}
    \cls[C]{S} = \Set[\geqslant2]{\clsAtom+\cls{K}+\cls[X]{S}}\text{.}
  \end{align*}
  This equation specifies that a subtree rooted at a star-node, linked to
  its parent by its center, has at least 2 unordered children attached to
  the extremities of the star-node: each extremity can either lead to a
  leaf, a regular clique-node, or another star-node entered through an
  extremity (but not another star-node entered through its center since
  the tree is reduced). The lower bound of 2 children is due to the fact
  that in a reduced split-decomposition tree, every internal node has degree at least 3,
  one of which is the star-node's center.

  Next,
  \begin{align*}
    \cls[X]{S} = \clsAtom\times\Set[\geqslant1]{\clsAtom+\cls{K}+\cls[X]{S}}\text{,}
  \end{align*}
  corresponding to a subtree rooted at a star-node, linked to its parent
  by an extremity. This star-node can be be exited either via its center
  and lead to a leaf $\clsAtom$ (the only type of element the center of a
  star-node can be connected to, following
  Theorem~\ref{thm:split-characterization-block}), or via some extremity
  and lead to a leaf, a regular clique-node, or another star-node entered
  through an extremity (but not another star-node entered through its
  center, as that is forbidden in reduced trees).
  
  Next, the equation corresponding to a clique-node,
  \begin{align*}
    \cls{K} = \Set[\geqslant2]{\clsAtom+\cls[X]{S}}\text{.}
  \end{align*}
  A clique-node has a degree of at least three, so a clique-rooted subtree
  can be exited from a set of at least two children and reach a leaf or a
  star-node through its extremity. It cannot reach another clique-node
  since the tree is reduced, and it cannot enter a star-node through its
  center since, again according to
  Theorem~\ref{thm:split-characterization-block}, star centers are only
  adjacent to leaves.

  Finally, this equation
  \begin{align*}
    \clsBGrl = \clsAtom_{\mLeaf} \times (\cls[C]{S}+\cls[X]{S}+\cls{K})\text{,}
  \end{align*}
  combines the previously introduced terms into a specification for rooted
  split-decomposition trees of block graphs, which are combinatorially equivalent to the
  class of rooted block graphs. It states that a rooted split-decomposition tree of a
  block graph consists of a distinguished leaf
  $\cramped{\clsAtom_{\mLeaf}}$, which is attached to an internal node.
  The internal node could be a clique-node, or a star-node entered through
  either its center or an extremity.
\end{proof}

\noindent With this symbolic specification, and a computer algebra system,
we may extract an arbitrary long enumeration (we've easily extracted
10\,000 terms).

\subsection{Unrooted grammar.%
  \label{subs:block-unrooted-grammar}}

Applying the dissymmetry theorem to the internal nodes and edges of
split-decomposition trees for block graphs gives the following grammar.

\begin{theorem}{\label{thm:unrooted-grammar-block}}
  The class $\clsBG$ of unrooted block graphs is specified by
  \begin{align}
  \clsBG          &= \cls[K]{T} +\cls[S]{T} + \cls[S-S]{T}
       - \cls[S\rightarrow S]{T} - \cls[S-K]{T} \label{eq:dissymmetry-simplified-block}\\[0.4em]
  \cls[K]{T}      &= \Set[\geqslant3]{\clsAtom+\cls[X]{S}}\\
  \cls[S]{T}      &= \clsAtom\times\cls[C]{S} \\
  \cls[S-S]{T}    &= \Set[2]{\cls[X]{S}}\\
  \cls[S\rightarrow S]{T}
                  &= \cls[X]{S}\times\cls[X]{S}\\
  \cls[S-K]{T}    &= \cls{K}\times\cls[X]{S}\\[0.4em]
  \cls[C]{S}      &= \Set[\geqslant2]{\clsAtom+\cls{K}+\cls[X]{S}}\\
  \cls[X]{S}      &= \clsAtom\times\Set[\geqslant1]{\clsAtom+\cls{K}+\cls[X]{S}}\\
  \cls{K}         &= \Set[\geqslant2]{\clsAtom+\cls[X]{S}}
\end{align}
\end{theorem}

\noindent As noted in Subsection~\ref{subs:dissymmetry}, in the unrooted
specification, the classes denoted by $\cramped{\cls[\ldots]{T}}$
correspond to trees introduced by the dissymmetry theorem, whereas the
specification of all other classes is identical to the rooted grammar for
block graph split-decomposition trees given in Theorem~\ref{thm:rooted-grammar-block}.

\begin{proof}
  From the dissymmetry theorem, we have the following bijection linking
  rooted and unrooted split-decomposition trees of block graphs,
  \begin{align*}
    \clsBG &= {{\clsBG}_{\mNode}} + 
               {{\clsBG}_{\mEdge}} -
               {{\clsBG}_{\mDEdge}}\text{.}
  \end{align*}
  \noindent Lemma~\ref{lem:no-leaves} allows us to consider only internal
  nodes for the rooted terms. Since block graphs are totally-decomposable
  into star-nodes and clique-nodes, we have the following symbolic
  equation for split-decomposition trees of block graphs rooted at an internal node,
  \begin{align*}
    {\clsBG}_{\mNode} &= \cls[K]{T} + \cls[S]{T}\text{.}
  \end{align*}
  
  \noindent Additionally, when rooting split-decomposition trees of block graphs at an
  undirected edge between internal nodes, the edge could either connect
  two star-nodes or a star-node and a clique-node (recall that clique-nodes
  cannot be adjacent in reduced trees by Theorem~\ref{thm:cunningham}),
  which yields the following symbolic equation for block graph split-decomposition trees
  rooted at an internal undirected edge,
  \begin{align*}
    {{\clsBG}_{\mEdge}} &= \cls[S-S]{T} + \cls[S-K]{T}\text{.}
  \end{align*}
  
  \noindent Finally, when rooting split-decomposition trees of block graphs at a
  directed edge between internal nodes, the edge could either go from a
  star-node to a clique-node, a clique-node to a star-node, or a star-node
  to another star-node (again, there are no adjacent clique-nodes by
  Theorem~\ref{thm:cunningham} of reduced trees), giving the following
  symbolic equation for block graph split-decomposition trees rooted at an internal
  directed edge,
  \begin{align*}
    {{\clsBG}_{\mDEdge}} &= \cls[S\rightarrow K]{T} + \cls[K\rightarrow S]{T} + \cls[S\rightarrow S]{T} \text{.}
  \end{align*}
  
  \noindent Combining the above equations with the dissymmetry equation
  for block graph split-decomposition trees gives
  \begin{align*}
    \clsBG &= \cls[K]{T} + \cls[S]{T}\\
           &+ \cls[S-S]{T} + \cls[S-K]{T}\\
           &- \cls[S\rightarrow K]{T} -
              \cls[K\rightarrow S]{T} -
              \cls[S\rightarrow S]{T} \text{.}
  \end{align*}
  
  \noindent We next observe the following bijection between ptolemaic
  trees rooted at an edge between a clique-node and a star-node,
  $\cls[S\rightarrow K]{T} \simeq \cls[K\rightarrow S]{T}$
  $ \simeq \cls[S-K]{T}$. This due to the fact that star- and clique-nodes
  are distinguishable, so an edge connecting a star-node and a clique-node
  bears an implicit direction. (One can, for example, define the direction
  to always be out of the clique-node into the star-node.) Simplifying
  accordingly, we arrive at
  Equation~\eqref{eq:dissymmetry-simplified-block},
  \begin{align*}
    \clsPG &= \cls[K]{T} +\cls[S]{T} + \cls[S-S]{T} - \cls[S\rightarrow S]{T} - \cls[S-K]{T}
  \end{align*}
  
  We will now discuss the symbolic equations for rooted split-decomposition trees of block graphs, starting with the following equation,
  \begin{align*}
     \cls[K]{T} = \Set[\geqslant3]{\clsAtom+\cls[X]{S}}\text{.}
  \end{align*} 
  \noindent
  This equation states that the split-decomposition tree of a block graph rooted at a clique-node can be specified as a set of at least three subtrees (since internal nodes in reduced split-decomposition trees have degree $\geqslant3$), each of which can lead to either a leaf or a star-node entered through its center; they cannot lead to clique-nodes as there are no adjacent clique-nodes in reduced split-decomposition trees, and they cannot lead to star-nodes through their centers, as centers of star-nodes in block graph split-decomposition trees only connect to leaves.

  Next, we will consider the equation,
  \begin{align*}
    \cls[S]{T}  = \clsAtom\times\cls[C]{S}\text{.}
  \end{align*}
  \noindent 
  which specifies a block graph split-decomposition tree rooted at a star-node. The specification of the subtrees of the distinguished star-node depends on whether they are connected to the center or an extremity of the root. The center of the root can only be attached to a leaf, while the subtrees connected to the extremities of the distinguished star-node are exactly those specified by an $\cls[C]{S}$.
  
  The other three rooted tree equations follow from with the same logic.
\end{proof}


\section{Ptolemaic graphs\label{sec:ptolemaic}}

\begin{figure*}
  \centering
  \includegraphics[scale=0.4]{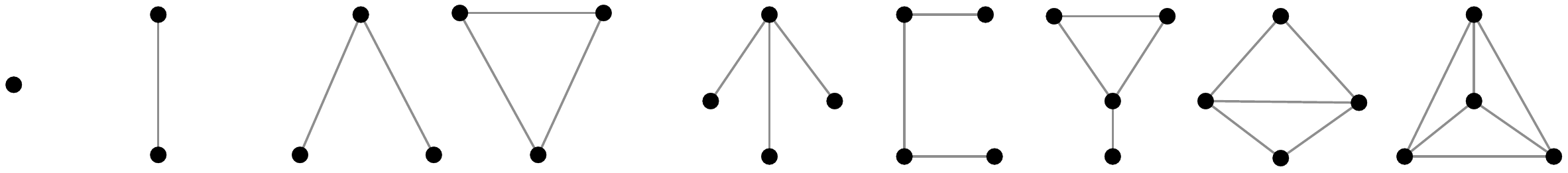}
  \caption{\label{fig:small-ptol}%
    Small (unrooted, unlabeled) ptolemaic graphs.}
\end{figure*}

Ptolemaic graphs were introduced by Kay and Chartrand~\cite{KaCh65} as the
class of graph that satisfied the same properties as a ptolemaic space.
Later, it was shown by Howorka~\cite{Howorka81} that these graphs are
exactly the intersection of distance-hereditary graphs and chordal graphs;
beyond that, relatively little is known about ptolemaic
graphs~\cite{UeUn09}, and in particular, their enumeration was hitherto
unknown.

\subsection{Characterization.}

\begin{definition}
  \label{def:graph-ptolemaic}%
  A graph $G$ is \emph{ptolemaic} if any four vertices $u,v,w,x$ in the
  same connected component satisfy the ptolemaic inequality~\cite{KaCh65}:
  \begin{align*}
    d_G(u,v) \cdot d_G(w,x) &\leqslant d_G(u,w) \cdot d_G(v,x) \\
                            &{}\quad + d_G(u,x) \cdot d_G(v,w).
  \end{align*}
  Equivalently, ptolemaic graphs are graphs that are both chordal and
  distance-hereditary~\cite[\S 2]{Howorka81}.
\end{definition}

\noindent This second characterization is the one that we will use:
indeed, by a reasoning similar to that provided in the proof of
Theorem~\ref{thm:split-characterization-block}, we have that
distance-hereditary graphs do not contain any $\cramped{C_{\geqslant 5}}$,
and chordal graphs do not contain any $\cramped{C_{\geqslant 4}}$; by
virtue of being a distance-hereditary graph (described by a clique-star
tree), we thus need only worry about the forbidden $\cramped{C_4}$ induced
subgraphs. As it so happens, we already have a characterization of a
split-decomposition tree which avoids such cycles.

\begin{theorem}[split-decomposition tree characterization of ptolemaic
  graphs\footnote{This characterization was given, but not proven, by
    Paul~\cite[p.~4]{Paul14} in an enlightening encyclopedia article
    related to split-decomposition.}]%
  \label{thm:split-characterization-ptol}%
  A graph $G$ with the reduced split-decomposition tree $(T,\mathcal{F})$ is ptolemaic
  if and only if
  \begin{enumerate}[label=(\alph*), noitemsep, nosep]
    \item $T$ is a clique-star tree;
    \item there are no center-center paths in $T$.
  \end{enumerate}
\end{theorem}

\begin{proof}
  Ptolemaic graphs are exactly the intersection of distance-hereditary
  graphs and chordal graphs. The first condition in this theorem addresses
  the fact that distance-hereditary graphs are exactly the class of
  totally decomposable with respect to the split-decomposition, and the
  second condition reflects the fact that, by
  Lemma~\ref{lem:split-characterization-c4}, center-center paths
  correspond to induced $\cramped{C_4}$ subgraphs, the defining forbidden
  subgraphs for chordal graphs.
\end{proof}

\subsection{Rooted grammar.%
  \label{subs:ptolemaic-rooted-grammar}}

Equipped with the characterization of a bijective split-decomposition tree
representation of ptolemaic graphs, we are now ready to enumerate
ptolemaic graphs. In this subsection, we begin by providing a grammar for
rooted split-decomposition trees of ptolemaic graphs, which can be used to enumerate
labeled ptolemaic graphs. Next, we derive the unlabeled enumeration.

\begin{theorem}{\label{thm:rooted-grammar}}
  The class $\clsPGrl$ of ptolemaic graphs rooted at a vertex is specified by
  \begin{align}
    \clsPGrl        &= \clsAtom_{\mLeaf} \times (\cls[C]{S}+\cls[X]{S}+\cls{K})\\
    \cls[C]{S}      &= \Set[\geqslant2]{\clsAtom+\cls{K}+\cls[X]{S}}\\
    \cls[X]{S}      &= (\clsAtom+\cls{\bar{K}})\times\Set[\geqslant1]{\clsAtom+\cls{K}+\cls[X]{S}}\\
    \cls{K}         &= \cls[C]{S}\times\Set[\geqslant1]{\clsAtom+\cls[X]{S}}+\Set[\geqslant2]{\clsAtom+\cls[X]{S}}\\
    \cls{\bar{K}}   &= \Set[\geqslant2]{\clsAtom+\cls[X]{S}}
  \end{align}
\end{theorem}

\begin{proof}
  The interesting part of this grammar is that, to impose the restriction
  on center-center paths (condition (b) of
  Theorem~~\ref{thm:split-characterization-ptol}), we must distinguish
  between two classes of clique-nodes, depending on the path through which
  we have reached them in the rooted tree:
  \begin{itemize}[noitemsep, nosep]
  \item $\cls{\bar{K}}$: these are clique-nodes for which the most recent
    star-node on their ancestorial path has been exited through its
    center; we call these clique-nodes \emph{prohibitive} to indicate that
    they cannot be connected to the center of a star-node;
  \item $\cls{K}$: all other clique-nodes, which we by contrast call
    \emph{regular}.
\end{itemize}

\noindent Recall that the split-decomposition tree of ptolemaic graphs
must, overall, satisfy the following constraints:
\begin{enumerate}[label=(\alph*), noitemsep, nosep]
\item center-center paths are forbidden
  (Theorem~\ref{thm:split-characterization-ptol});
\item internal nodes must have degree at least 3
  (Thm.~\ref{thm:cunningham});
\item the center of a star-node cannot be incident to the extremity of
  another star-node (Theorem~\ref{thm:cunningham});
\item two clique-nodes cannot be adjacent (Theorem~\ref{thm:cunningham}).
\end{enumerate}

\noindent We can now prove the correctness of the grammar. We begin with
the following equation
\begin{align*}
  \cls[C]{S} = \Set[\geqslant 2]{\clsAtom+\cls{K}+\cls[X]{S}}
\end{align*}
which specifies that a subtree rooted at a star-node, linked to its parent
by its center, has at least 2 unordered children as the extremities of the
star-node: each extremity can either lead to a leaf, a regular
clique-node, or another star-node entered through an extremity. The
children subtrees cannot be star-nodes entered through their center, since
the tree is reduced. The lower bound of two children is due to the first
condition of reduced split-decomposition trees (Theorem~\ref{thm:cunningham}), which
specifies that every internal node has degree at least 3.

We now consider the next equation
\begin{align*}
  \cls[X]{S} = (\clsAtom+\cls{\bar{K}})\times\Set[\geqslant1]{\clsAtom+\cls{K}+\cls[X]{S}}
\end{align*}
The disjoint union in this equation indicates that a subtree rooted at a
star-node, linked to its parent by an extremity, can be exited in two
ways, either through the center, or through another extremity.

If the star-node is exited through its center, it can either enter a leaf
or a prohibitive clique-node. It cannot enter a $\cls[X]{S}$ by the third
condition of reduced split-decomposition trees (Theorem~\ref{thm:cunningham}), and it
cannot enter a $\cls[C]{S}$ as that would be a center-center path.
Furthermore, it has to enter a prohibitive clique-node $\cls{\bar{K}}$
rather than a regular clique-node $\cls{K}$ to keep track of the fact that
a star-node has been exited from its center on the current path and ensure
that no another star-node will not be entered through its center.

If the star-node entered from an extremity is exited through an extremity,
it has a set of at least one other extremity to choose from. Each of those
extremities can lead to a either a leaf, a clique-node, or another
star-node entered through its center. It cannot lead to an $\cls[C]{S}$,
as that would be a center-center path.

We next discuss the equation
\begin{align*}
  \cls{K} = \cls[C]{S}\times\Set[\geqslant1]{\clsAtom+\cls[X]{S}}+\Set[\geqslant2]{\clsAtom+\cls[X]{S}}
\end{align*}
The disjoint union specifies that a subtree rooted at a regular clique-node can have exactly zero or one $\cls[C]{S}$ as a child. First, a regular-clique-rooted subtree is allowed to have a $\cls[C]{S}$ as a child, since regular clique-nodes are by definition not on potential center-center paths. However, a regular-clique-rooted subtree cannot have more than one $\cls[C]{S}$ child, since otherwise there would be a center-center path between the $\cls[C]{S}$ children through the clique-node.

The first summand corresponds to the case where the regular clique-node has exactly one $\cls[C]{S}$ as a child, which can be used to exit the tree. Additionally, the clique-node can be exited via any of the other children besides $\cls[C]{S}$ and reach either a leaf or a star-node entered through an extremity. Note that the clique-node cannot be exited into another clique-node of any kind by the second condition of reduced split-decomposition trees (Theorem~\ref{thm:cunningham}), which indicates that no two clique-nodes are adjacent in a reduced split-decomposition tree.

The second summand corresponds to the case where the regular clique-node has no $\cls[C]{S}$ children. In this case, the regular clique-node can be exited via any of the remaining two or more subtrees that have not been used to enter it. After exiting the clique-node, one arrives at either a leaf or a star-node entered through its extremity. As explained above, is not possible to arrive at a clique-node, since there are no adjacent clique-nodes in reduced split-decomposition trees.

We now take a look at the equation specifying subtrees rooted at prohibitive clique-nodes
\begin{align*}
  \cls{\bar{K}} = \Set[\geqslant2]{\clsAtom+\cls[X]{S}}
\end{align*}
A subtree rooted at a prohibitive clique-node can be exited via any of its set of at least two children and either enter a leaf or enter a star-node through its extremity. Since a prohibitive clique-nodes lies on a path from the center of a star-node, it cannot enter a $\cls[C]{S}$. Additionally, it cannot enter another clique-node of any kind since reduced clique-nodes cannot be adjacent.

Finally, the following equation
\begin{align*}
  \clsPGrl = \clsAtom_{\mLeaf}\times(\cls[X]{S}+\cls[C]{S}+\cls{K})
\end{align*}
combines all pieces into a symbolic specification for rooted ptolemaic graphs. It states that a rooted ptolemaic graph consists of a distinguished leaf $\cramped{\clsAtom_{\mLeaf}}$, which is attached to an internal node. The internal node could be star-node entered through either its center or an extremity, or it could be a regular clique-node.

\end{proof}

\noindent Given this grammar for ptolemaic graphs, we can produce the
exact enumeration for rooted labeled ptolemaic graphs using a computer
algebraic system. Furthermore, we can derive the enumeration of
\emph{unrooted} labeled ptolemaic graphs by normalizing the counting
sequence by the number of possible ways to distinguish a vertex as the
root. This normalization is easy for labeled graphs, since the labels
prevent the formation of symmetries. Therefore, since each vertex is
equally likely to be chosen as the root, the number of \emph{unrooted}
labeled graphs of size $n$ is simply the number of \emph{rooted} labeled
graphs divided by $n$.

\subsection{Unrooted grammar.%
  \label{subs:ptolemaic-unrooted-grammar}}

\begin{theorem}{\label{thm:unrooted-grammar-ptol}}
  The class $\clsPG$ of unrooted ptolemaic graphs is specified by
  \begin{align}
  \clsPG          &= \cls[K]{T} +\cls[S]{T} + \cls[S-S]{T} - \cls[S\rightarrow S]{T} - \cls[S-K]{T} \label{eq:dissymmetry-simplified-ptol}\\
  \cls[K]{T}      &= \cls[C]{S}\times\Set[\geqslant2]{\clsAtom+\cls[X]{S}} + \Set[\geqslant3]{\clsAtom+\cls[X]{S}}\\
  \cls[S]{T}      &= \cls[C]{S}\times(\clsAtom+\cls{\bar{K}}) \\
  \cls[S-S]{T}    &= \Set[2]{\cls[X]{S}}\\
  \cls[S\rightarrow S]{T}
                  &= \cls[X]{S}\times\cls[X]{S}\\
  \cls[S-K]{T}    &= \cls{K}\times\cls[X]{S} + \cls{\bar{K}}\times\cls[C]{S}\\
  \cls[C]{S}      &= \Set[\geqslant 2]{\clsAtom+\cls{K}+\cls[X]{S}}\\
  \cls[X]{S}      &= (\clsAtom+\cls{\bar{K}})\times\Set[\geqslant1]{\clsAtom+\cls{K}+\cls[X]{S}}\\
  \cls{K}         &= \cls[C]{S}\times\Set[\geqslant1]{\clsAtom+\cls[X]{S}}+\Set[\geqslant2]{\clsAtom+\cls[X]{S}}\\
  \cls{\bar{K}}   &= \Set[\geqslant2]{\clsAtom+\cls[X]{S}}
\end{align}
\end{theorem}

\begin{proof}
  Applying the dissymmetry theorem in a similar manner to the proof of the
  unrooted grammar of block graphs, we obtain the following formal
  equation:
  \begin{align*}
    \begin{split}
      \clsPG &= \cls[K]{T} + \cls[S]{T} + \cls[S-S]{T} + \cls[S-K]{T}\\
             &- \cls[S\rightarrow K]{T} - \cls[K\rightarrow S]{T} -
                \cls[S\rightarrow S]{T}\text{.}
    \end{split}
  \end{align*}
  Notably, even though we distinguish between prohibitive $\cls{\bar{K}}$
  and regular $\cls{K}$ clique-nodes in the rooted grammar, this
  distinction disappears when rerooting the trees for the dissymmetry
  theorem. This is because the prohibitive or regular nature of a
  clique-node depends on an \emph{implicitly directed} path leading to it
  from the root; however when rerooting the tree, the clique-node in
  question becomes the new root, and all (implicitly directed) paths
  originate from it\footnote{This notion is implicitly used in the
    unrooted grammar for block graphs---and previously by
    Chauve~\etal~\cite{ChFuLu14}, for distance-hereditary and 3-leaf power
    graphs---in which we only reroot at a star-node $\cls{S}$ without
    distinguishing whether it was entered by its center or an extremity,
    precisely because it is the new root, and therefore all paths lead
    away from it.}.
  
  We can then simplify the dissymmetry theorem equation above in the same
  manner as for block graphs:
  \begin{align*}
    \clsPG &= \cls[K]{T} +\cls[S]{T} + \cls[S-S]{T}
             - \cls[S\rightarrow S]{T} - \cls[S-K]{T}\text{.}
  \end{align*}
  We then discuss the rerooted terms of the grammar, starting with the
  following equation,
  \begin{align*}
    \cls[K]{T} &= \cls[C]{S}\times\Set[\geqslant2]{\clsAtom+\cls[X]{S}} +
                  \Set[\geqslant3]{\clsAtom+\cls[X]{S}}\text{.}
  \end{align*} 
  The disjoint union translates the fact that in a ptolemaic tree rerooted
  at a clique-node, the root can have either zero or one $\cls[C]{S}$ as a
  subtree. (A clique-node having more than one $\cls[C]{S}$ subtree would
  induce a center-center path, which cannot exist in ptolemaic split-decomposition trees
  by Theorem~\ref{thm:split-characterization-ptol}.)
  
  The first summand corresponds to the case where the clique-node at which
  the split-decomposition tree is rooted has one $\cls[C]{S}$ as a subtree. The
  clique-node root can have a set of at least two other subtrees, each of
  which can lead to either a leaf or a star-node entered from an
  extremity.
  
  The second summand corresponds to the case where the clique-node at
  which the split-decomposition tree is rooted has no $\cls[C]{S}$ subtrees, in which
  case it can have a set of at least three other subtrees leading to
  leaves or $\cls[X]{S}$ nodes, but not clique-nodes or $\cls[C]{S}$ as
  explained.
  
  Next, we will consider the equation,
  \begin{align*}
    \cls[S]{T} &= \cls[C]{S}\times(\clsAtom+\cls{\bar{K}})
  \end{align*}
  which specifies a ptolemaic split-decomposition tree rooted at a star-node. The
  specification of the subtrees of the distinguished star-node depends on
  whether they are connected to the center or an extremity of the
  star-node. The subtrees connected to the extremities of the
  distinguished star-node are exactly those specified by an $\cls[C]{S}$,
  and the subtree attached to the center of the distinguished star-node
  can lead to either a leaf or a prohibitive clique-node.
  
  The other three rooted tree equations follow from with the same logic.
\end{proof}

\noindent The first few terms of the enumeration of unlabeled, unrooted
ptolemaic graphs (among others) are available in the
Table~\ref{tab:enum-block-ptol} at the end of this paper.

\section{2,3-Cactus, 3-Cactus and 4-Cactus Graphs%
  \label{sec:cactus}}

The definition of \emph{cactus graphs} is similar to that of a block
graphs. Yet whereas in block graphs (discussed in
Section~\ref{sec:block}), the blocks\footnote{Recall that a block, or
  \emph{biconnected component}, is a maximal subgraph in which every two
  vertex, or every edge belongs to a simple cycle.} are cliques, in a
cactus graph the blocks are cycles. Thus, just as block graphs can be
called clique trees, cacti can be seen as ``\emph{cycle
  trees}''\footnote{Although cactus graphs have been known by many
  different names, including \emph{Husimi Trees} (a term that grew
  contentious because the graphs are not in fact trees~\cite[\S
  3.4]{HaPa73}---although this seems not to have been an issue for
  $k$-trees and related classes!), they have not generally been known by
  the name ``\emph{cycle trees}'', except in a non-graph theoretical
  publication, which rediscovered the concept~\cite{FrJo83}.}. An
alternate definition:

\begin{definition}
  A \emph{cactus} is a connected graphs in which every edge belongs to at
  most one cycle~\cite{HaPa73}
\end{definition}

\noindent We can also conjure further variations on this definition, with
cactus graphs having as blocks, cycles that have size constrained to a set
of positive integers; thus given a set of integers $\Omega$, an
$\Omega$-cactus graph\footnote{Note that it makes no sense for 1 to be in
  $\Omega$ given this definition. We can however have $2\in\Omega$, in
  which case we treat an edge as a cycle of size 2. For example, if
  $2\not\in\Omega$, every vertex must be part of a cycle.} is the class of
cactus graphs of which the cycles have size $m \in \Omega$. In this
section, we discuss cactus graphs for the sets: $\Omega = \set{2,3}$,
$\Omega = \set{3}$ and $\Omega = \set{4}$, following an article by Harary
and Uhlenbeck~\cite{HaUh53}, who use dissimilarity characteristics derived
from Otter's theorem~\cite{Otter48}.


\subsection{2,3-Cactus Graphs.\label{subs:2-3-cactus}}

In this section, we enumerate the family of cactus graphs with
$\Omega=\{2,3\}$. The class of 2,3-cactus graphs is equivalent to the
intersection of block graphs and of cactus graphs\footnote{Block graphs
  can be thought of as a set of cliques sharing at most one vertex
  pairwise, and cactus graphs can be thought of as a set of cycles sharing
  at most one vertex pairwise. The intersection of cycles and cliques are
  those of sizes 1, 2, and 3; however, in the case of one vertex, adding a
  single vertex in this manner to a connected block or cactus graph does
  not change the size of the graph, contradicting the requirement that in
  a combinatorial class, there must be a finite number of objects of any
  fixed size. Therefore, the intersection of block graphs and cactus
  graphs is the family of 2,3-cactus graphs.}, not to be confused with the
class of block-cactus graphs (which are the \emph{union} of block graphs
and cactus graphs~\cite{RaVo98}).


\begin{theorem}[split-decomposition tree characterization of 2,3-cactus
  graphs]{\label{thm:2-3-cactus-split-characterization}}
  A graph $G$ with the reduced split-decomposition tree $(T,\mathcal{F})$ is a
  block-cactus graph if and only if
  \begin{enumerate}[label=(\alph*), noitemsep, nosep]
    \item $T$ is a clique-star tree;
    \item every clique-node has degree 3;
    \item the center of all star-nodes are attached to leaves;
  \end{enumerate}
\end{theorem}

\begin{proof}
  This split-decomposition tree characterization is identical to the characterization
  for 3-cactus graphs, except the last condition in the characterization
  of 3-cacti (stating that leaves cannot be attached to extremities of
  star-nodes) is missing here. As we outlined in the proof of the
  characterization of 3-cacti, a leaf attached to a an extremity of an
  star-node corresponds to a vertex of degree 1 in the original graph.
  Unlike with 3-cacti, which required that all vertices be in some cycle
  of size 3, having such a vertex of degree 1 here corresponds to a
  $\cramped{C_2}$ and is allowed. Therefore, the correctness of this
  characterization follows from the proof of the characterization for
  3-cactus graphs.
\end{proof}

\begin{theorem}{\label{thm:rooted-grammar-2-3-cactus}}
  The class $\clsTTCGrl$ of 2,3-cactus graphs rooted at a vertex is
  specified by
  \begin{align}
      \clsTTCGrl      &= \clsAtom_{\mLeaf} \times (\cls[C]{S}+\cls[X]{S}+\cls{K})\\
      \cls[C]{S}      &= \Set[\geqslant2]{\clsAtom+\cls{K}+\cls[X]{S}}\\
      \cls[X]{S}      &= \clsAtom\times\Set[\geqslant1]{\clsAtom+\cls{K}+\cls[X]{S}}\\
      \cls{K}          &= \Set[=2]{\clsAtom+\cls[X]{S}}
    \end{align}
\end{theorem}

\begin{theorem}{\label{thm:unrooted-grammar-2-3-cactus}}
  The class $\clsTTCG$ of unrooted 2,3-cactus graphs is specified by
  \begin{align}
  \clsTTCG        &= \cls[K]{T} +\cls[S]{T} + \cls[S-S]{T} - \cls[S\rightarrow S]{T} - \cls[S-K]{T} \\
  \cls[K]{T}       &= \Set[=3]{\clsAtom+\cls[X]{S}}\\
  \cls[S]{T}       &= \clsAtom\times\cls[C]{S} \\
   \cls[S-S]{T}   &= \Set[2]{\cls[X]{S}}\\
  \cls[S\rightarrow S]{T}
                        &= \cls[X]{S}\times\cls[X]{S}\\
  \cls[S-K]{T}    &= \cls{K}\times\cls[X]{S}\\
  \cls[C]{S}       &= \Set[\geqslant2]{\clsAtom+\cls{K}+\cls[X]{S}}\\
  \cls[X]{S}       &= \clsAtom\times\Set[\geqslant1]{\clsAtom+\cls{K}+\cls[X]{S}}\\
  \cls{K}           &= \Set[=2]{\clsAtom+\cls[X]{S}}
\end{align}
\end{theorem}

\subsection{3-Cactus Graphs\label{subs:3-cactus}}

We now enumerate the family of cactus graph that is constrained to
$\Omega=\{3\}$, which we refer to as the family as 3-cacti or
$\emph{triangular}$ cacti.

\begin{figure*}
  \centering
  \includegraphics[scale=0.4]{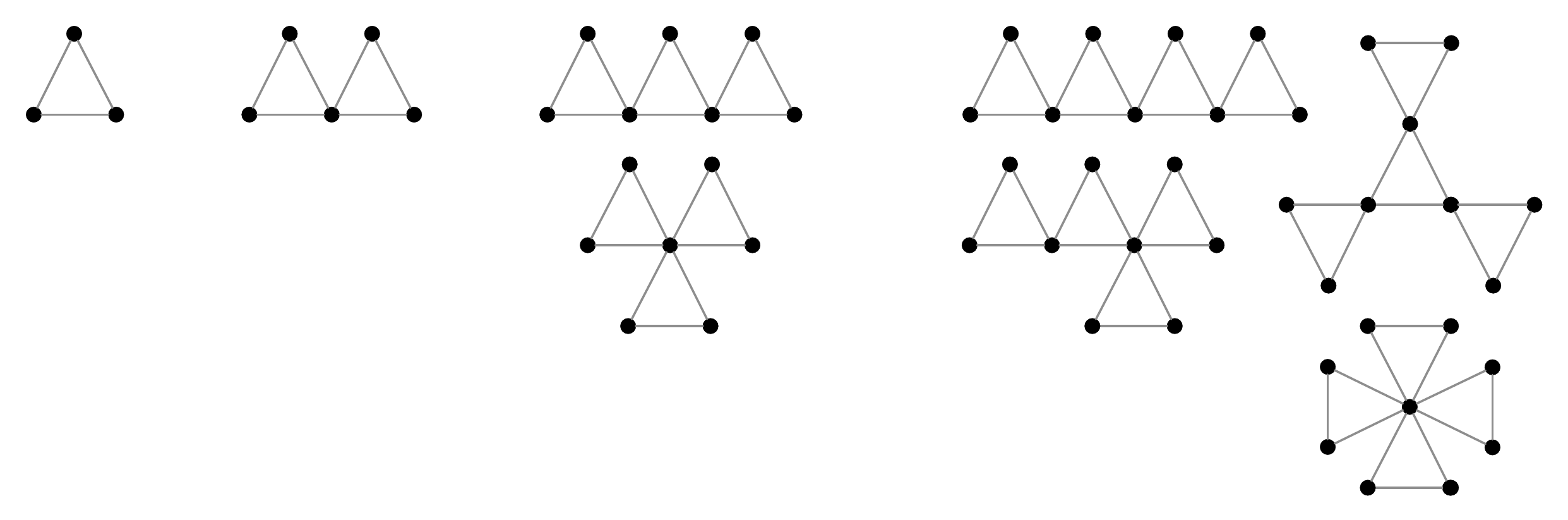}
  \caption{\label{fig:small-3cacti} Small (unrooted, unlabeled) triangular cacti.}
\end{figure*}

\begin{lemma}[Forbidden subgraph characterization of 3-cacti]\label{lem:forbidden-characterization-3cactus}
A graph $G$ is a triangular cactus if and only if $G$ is a block graph with no bridges or induced $\cramped{K_{\geqslant4}}$.
\end{lemma}

\begin{proof}
  \noindent {[$\Rightarrow$]}~~Given a triangular cactus $G$, we will show
  that $G$ is a block graph and does not have any bridges or induced
  $\cramped{K_{\geqslant4}}$.

  We first note that $G$ is a block graphs by showing that it is
  ($\cramped{C_{\geqslant4}}$, diamond)-free. There cannot be any induced
  $\cramped{C_{\geqslant4}}$ in $G$, because every edge of a 3-cactus is
  in exactly one triangle and no other cycle. There cannot be any induced
  diamonds in a $G$ because diamonds have an edge in common between two
  cycles\footnote{Here is another way to see why 3-cacti are a subset of
    block graphs. Block graphs can be thought of a set of cliques sharing
    at most one vertex pairwise, and cactus graphs can be thought of a set
    of cycles sharing at most one vertex pairwise. Since triangles are
    both cycles and cliques, a pairwise edge-disjoint collection of them
    is both a cactus graph and block graph.}.

  We next observe that $G$ cannot have bridges, as a bridge is by
  definition not part of any cycles, including triangles. Furthermore, $G$
  cannot have any induced cliques on 4 or more vertices, as such a clique
  would involve edges shared between triangles. Therefore, $G$ must be a
  block graph and with no pendant edges or induced
  $\cramped{K_{\geqslant4}}$.

  
  \noindent {[$\Leftarrow$]}~~Given a block graph $G$ without any bridges
  or induced $\cramped{K_{\geqslant4}}$, we need to show that $G$ is a
  3-cactus. We do so by showing that every edge $(a,b)\in E(G)$ is in
  exactly one triangle and no other cycle.
      
  First, since $(a,b)$ cannot be a bridge, it must lie on some cycle $C$.
  Since $G$ is a block graph and thus $\cramped{C_{\geqslant4}}$-free, $C$
  must be a triangle.Furthermore, if $(a,b)$ belonged to another cycle
  $C'$, by the same argument, $C'$ would also be a triangle. Let $c$ be
  the third vertex of $C$ other than $a$ and $b$, and let $c'$ the third
  vertex of $C'$. Depending on the adjacency of $c$ and $c'$, we have one
  of the following two cases:
  \begin{itemize}[noitemsep, nosep]
  \item $(c,c')\in E(G)$, in which case $\set{a,c,b,c'}$ induces a
    $\cramped{K_4}$, which we assumed $G$ does not include;
  \item $(c,c')\not\in E(G)$, in which case $\set{a,c,b,c'}$ induces a
    diamond, which $G$, as a block graph, cannot contain.
  \end{itemize}
  Therefore, no such cycle $C'$ can exist, implying that $(a,b)$ belongs
  to one and exactly one triangle in $G$ and no other cycle. Extending
  this argument to all edges of $G$ ensures that $G$ is a 3-cactus.
\end{proof}

\begin{theorem}[split-decomposition tree characterization of 3-cacti]%
  \label{thm:split-characterization-3cacti}%
  A graph $G$ with the reduced split-decomposition tree $(T,\mathcal{F})$ is a
  triangular cactus graph if and only if
  \begin{enumerate}[label=(\alph*), noitemsep, nosep]
    \item $T$ is a clique-star tree;
    \item the centers of all star-nodes are attached to leaves;
    \item the extremities of star-nodes are only attached to clique-nodes;
    \item every clique-node has degree 3.
  \end{enumerate}
\end{theorem}

\begin{proof}
  By Lemma~\ref{lem:forbidden-characterization-3cactus}, we know that
  3-cacti can be described exactly as the class of block graphs with no
  bridges or induced $\cramped{K_\geqslant4}$.
  
  The first and second conditions of this theorem duplicate the split-decomposition tree
  characterization of block graphs outlined in
  Theorem~\ref{thm:split-characterization-block}.

  The third condition uses Lemma~\ref{lem:split-characterization-bridge}
  to forbid bridges. Since by the second condition, all star centers in
  $T$ are adjacent to leaves, a star extremity adjacent to a leaf of $T$
  would correspond to a bridge in the form of a pendant edge in $G$, and a
  star extremity adjacent to another star extremity would correspond to a
  non-pendant bridge in $G$.

  Finally, the last condition applies
  Lemma~\ref{lem:split-characterization-K4} to disallow
  $\cramped{K_\geqslant4}$, the last set of forbidden induced subgraphs
  for 3-cactus.
\end{proof}

\noindent The split-decomposition tree characterization of 3-cacti derived in the
previous section naturally defines the following symbolic grammar for
rooted block graphs.

\begin{theorem}{\label{thm:rooted-grammar-3cactus}}
  The class $\clsTCGrl$ of triangular cactus graphs rooted at a vertex is
  specified by
  \begin{align}
      \clsTCGrl        &= \clsAtom_{\mLeaf} \times (\cls[C]{S}+\cls{K})\\
      \cls[C]{S}      &= \Set[\geqslant2]{\cls{K}}\\
      \cls[X]{S}      &= \clsAtom\times\Set[\geqslant1]{\cls{K}}\\
      \cls{K}         &= \Set[=2]{\clsAtom+\cls[X]{S}}
    \end{align}
\end{theorem}

\begin{theorem}{\label{thm:unrooted-grammar-3cactus}}
  The class $\clsTCG$ of unrooted triangular cactus graphs is specified by
  \begin{align}
  \clsTCG          &= \cls[K]{T} +\cls[S]{T} - \cls[S-K]{T} \\
  \cls[K]{T}        &= \Set[=3]{\clsAtom+\cls[X]{S}}\\
  \cls[S]{T}        &= \clsAtom\times\cls[C]{S} \\
  \cls[S-K]{T}    &= \cls{K}\times\cls[X]{S}\\
  \cls[C]{S}       &= \Set[\geqslant2]{\cls{K}}\\
  \cls[X]{S}       &= \clsAtom\times\Set[\geqslant1]{\cls{K}}\\
  \cls{K}           &= \Set[\geqslant2]{\clsAtom+\cls[X]{S}}
\end{align}
\end{theorem}


\section{Conclusion\label{sec:conclusion}}

In this paper, we follow the ideas of Gioan and Paul~\cite{GiPa12} and
Chauve~\etal~\cite{ChFuLu14}, and provide full analyses of several
important subclasses of distance-hereditary graph. Some of these analyses
have lead us to uncover previously unknown enumerations (ptolemaic graphs,
...), while for other classes for which enumerations were already known
(block graphs, 2,3-cactus and 3-cactus graphs), we have provided symbolic
grammars which are a more powerful starting point for future work: such as
parameter analyses, exhaustive and random generation and the empirical
analyses that the latter enables, etc.. For instance, Iriza~\cite[\S
7]{Iriza15} provided a nice tentative preview of the type of results
unlocked by these grammars, when he empirically observed the linear growth
of clique-nodes and star-nodes in the split-decomposition tree of a random
distance-hereditary graph.

Our main idea is encapsulated in Section~\ref{sec:forbidden}: we think
that the split-decomposition, coupled with analytic combinatorics, is a
powerful way to analyze classes of graphs specified by their forbidden
induced subgraph. This is remarkably noteworthy, because forbidden
characterizations are relatively common, and yet they generally are very
difficult to translate to specifications. What we show is that this can be
(at least for subclasses of distance-hereditary graphs which are totally
decomposable by the split-decomposition) fairly automatic, in keeping with
the spirit of analytic combinatorics:
\begin{enumerate}[label=(\roman*)]
\item identify forbidden induced subgraphs;
\item translate each forbidden subgraph into constraints on the
  (clique-star) split-decomposition tree;
\item describe rooted grammar, apply unrooting, etc..
\end{enumerate}
This allows us to systematically derive the grammar of a number of
well-studied classes of graphs, and to compute full enumerations,
asymptotic estimates, etc.. In Figure~\ref{fig:ratios-subsets}, for
instance, we have used the results from this paper to provide some
intuition as to the relative ``density'' of these graph classes. A fairly
attainable goal would be to use the asymptotics estimates which can be
derived automatically from the grammars, to compute the asymptotic
probability that a random block graph is also a 2,3-cactus graph.\bigskip

\noindent Naturally, this raises a number of interesting questions, but
possibly the most natural one to ask is: can we expand this methodology
beyond distance-hereditary graphs, to classes for which the
split-decomposition tree contains \emph{prime nodes} (which are neither
clique-nodes nor star-nodes).

Beyond distance-hereditary graphs, another perfect (pun intended)
candidate is the class of \emph{parity graphs}: these are the graphs whose
split-decomposition tree has prime nodes that are bipartite graphs. But
while bipartite graphs have been enumerated by Hanlon~\cite{Hanlon79}, and
more recently Gainer-Dewar and Gessel~\cite{GaGe14}, it is unclear whether
this is sufficient to derive a grammar for parity graphs. Indeed, the
advantage of the degenerate nodes (clique-nodes and star-nodes) is that
their symmetries are fairly uncomplicated (all the vertices of a clique
are undistinguished; all the vertices of a star, save the center, are
undistinguished), as is in fact their enumeration (for each given size,
there is only one clique or one star). An empirical study by
Shi~\cite{ShLu15} showed that lower and upper bounds can be derived by
plugging in the enumeration as an artificial generating function---either
assuming all vertices of a bipartite prime node to be distinguished or
undistinguished.

Other classes present a similar challenge, in that the subset of allowable
prime nodes is itself too challenging.

A likely more fruitful direction to pursue this work is to first start
with classes of graphs which have small, predictable subsets of prime
nodes. We discovered one such family of classes in a paper by Harary and
Uhlenbeck~\cite{HaUh53}; in this paper, they discuss the enumeration of
unlabeled and unrooted 3-cactus graphs and 4-cactus graphs (which we
studied and enumerated using a radically different methodology in this
paper), while suggesting that they would have liked to provide some
general methodology to obtain the enumeration of $m$-cactus graphs, for
generalized polygons on $m$ sides\footnote{It seemsthat Harary and
  Uhlenbeck have never published such a paper; and it appears that the
  closest there is in terms of a general enumeration of $m$-cactus graphs
  is by Bona~\etal~\cite{BoBoLaLe99}---yet they enumerate graphs which are
  embedded in the plane, while we seek to enumerate the non-plane,
  unlabeled and unrooted $m$-cactus graphs.}.

There is some evidence to suggest that $m$-cactus graphs would yield
split-decomposition trees with prime nodes that are undirected cycles of
size $m$; likewise the split-decomposition tree of a general cactus graph
(in which the blocks are cycles of any size larger than 3) would likely
have prime nodes that are undirected cycles. In the same vein,
block-cactus graphs (which are the union of the block graphs enumerated in
Section~\ref{sec:block}, and of generalized cactus graphs) would likely
also have the same type of prime nodes. All of these are more manageable
subset, and it is likely that the various intersection classes with cactus
graphs would be a more promising avenue by which to determine whether the
split-decomposition can be reliably used for the enumeration of supersets
of distance-hereditary graphs.


\begin{table*}[p]
  \centering
  \def\arraystretch{1.5}
  \begin{tabular}{lr}
    \toprule
    \textbf{Graph Class}& \textbf{Number}\\
    \midrule
    General~\cite[A000088]{Sloane}& $2.86 \times \cramped{10^{685}}$\\
    General Connected~\cite[A001349]{Sloane}&$2.86 \times \cramped{10^{685}}$\\
    Distance-Hereditary~\cite{ChFuLu14}& $3.38 \times \cramped{10^{56}}$\\
    3-Leaf Power~\cite{ChFuLu14}& $8.40 \times \cramped{10^{37}}$\\
    Ptolemaic                  & $3.78 \times \cramped{10^{50}}$\\
    Block                         & $1.44 \times \cramped{10^{40}}$\\
    2,3-Cacti                  & $1.55 \times \cramped{10^{38}}$\\
    3-Cacti                    & $9.13 \times \cramped{10^{16}}$\\
    4-Cacti                    & $5.73 \times \cramped{10^{14}}$\\
    \bottomrule
  \end{tabular}
  \caption{%
    \label{tab:sample-enum}%
    Number of unlabeled graphs of size $n = 73$ for different
    graph classes.}
\end{table*}

\begin{figure*}[p]
  \centering
  \begin{bigcenter}
    \begin{minipage}[t]{.32\linewidth}
      \centering
      \includegraphics[scale=0.5]{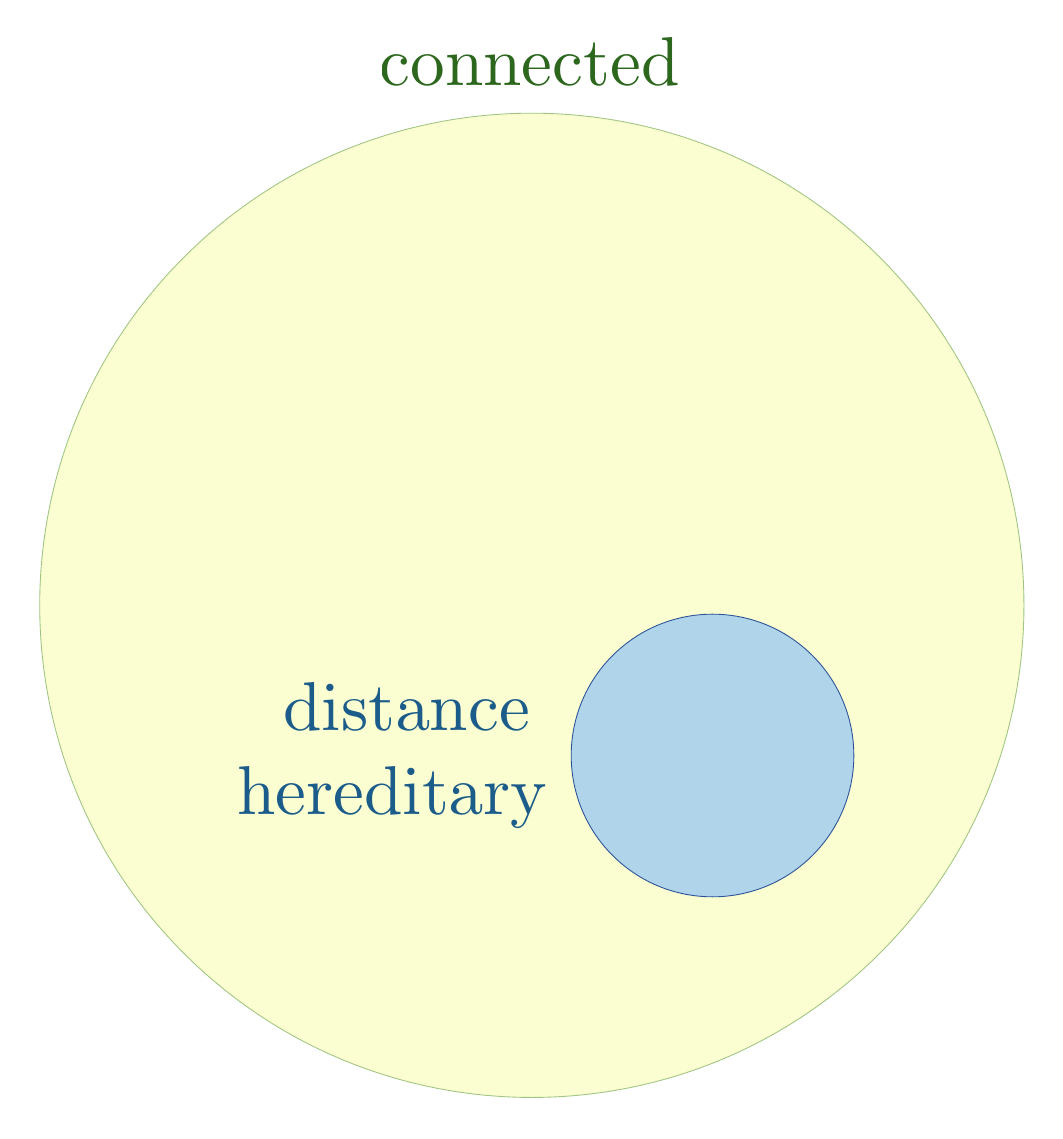}%
      \subcaption{%
        \label{fig:ratios-subsets-a}%
        Ratio of the size of distance-hereditary graphs to general
        connected graphs on a logarithmic scale.}
    \end{minipage}\hspace{2em}
    \begin{minipage}[t]{.60\linewidth}
      \centering
      \includegraphics[scale=0.5]{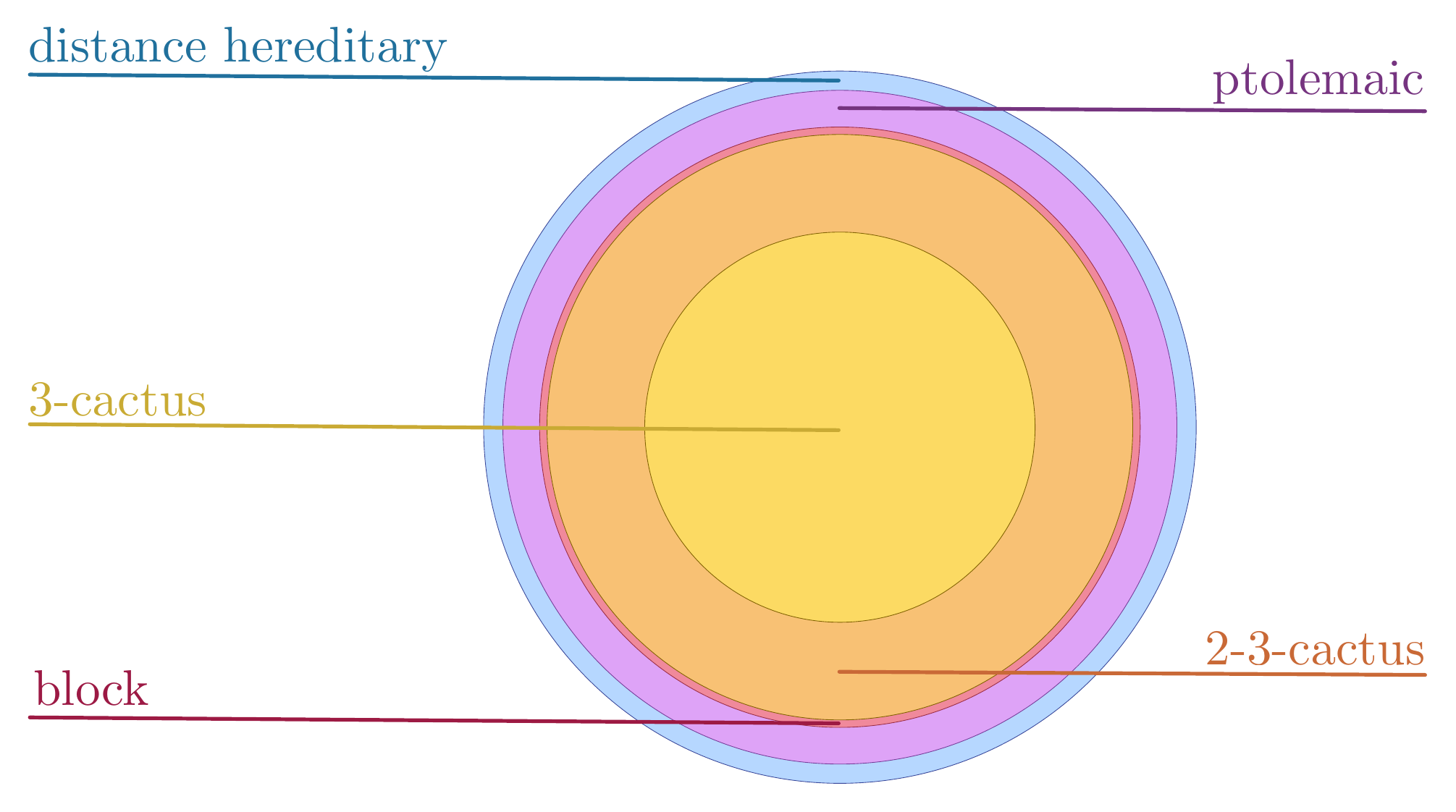}%
      \subcaption{%
        \label{fig:ratios-subsets-b}%
        Ratios of the size of various subsets of distance-hereditary
        graphs on a logarithmic scale.}
    \end{minipage}
  \end{bigcenter}
  \caption{%
    \label{fig:ratios-subsets}%
    Illustration of the ratios of sizes of various graph classes, for
    $n=73$ as depicted in Table~\ref{tab:sample-enum}. The radii are the
    square roots of the ratios of the logarithms of the enumeration for a
    given class to the logarithm of the enumeration for the base class. In
    addition, only strict subsets are displayed on the right. For
    instance, 3-leaf power graphs are a subset of ptolemaic graphs, since
    their characterization~\cite[\S 3.3]{GiPa12} does not allow for
    center-center paths; but they are not a subset of block graphs.
    Similarly, we cannot represent 4-cacti as they have no intersection
    with 3-cacti, and it is not physically possible to represent them here
    given our logarithmic scale.}
\end{figure*}


\begin{table*}
  \centering
  \def\arraystretch{1.2}
  \begin{tabularx}{0.84\textwidth}{cW}
    \toprule
    \textbf{Symbol}&\textbf{Explanation}\\
    \midrule
    $\cls{K}$ &{a clique-node entered from one of its vertices (and missing
                the corresponding subtree)}\\
    $\cls[C]{S}$ &{a star-node entered through its \emph{center} (and
                   missing the corresponding subtree)}\\
    $\cls[X]{S}$ &{a star-node entered through one of its (at least two)
                   \emph{extremities} (and missing the corresponding
                   subtree)}\\[0.4em]
    $\clsAtom$ &{a leaf of the split-decomposition tree
                 (an atom with unit size)}\\
    $\clsAtom_\mLeaf$ &{the \emph{rooted} leaf of the
                       {split-decomposition tree}
                       (an atom with unit size)}\\[0.4em]
    \midrule
    $\cls[K]{T}$ &{a {split-decomposition tree} rerooted at a
                   clique-node (all subtrees are present)}\\
    $\cls[S]{T}$ &{a {split-decomposition tree} rerooted at a
                   star-node (all subtrees are present)}\\[0.4em]
    $\cls[K-S]{T}$ &{a {split-decomposition tree} rerooted at an
                     \emph{edge} connecting a clique-node to a star-node
                     (the edge can either connect the clique-node to the
                     star-node's center or an extremity; the edge accounts
                     for one subtree of the clique-node and one subtree
                     of the star-node)}\\
    $\cls[S-S]{T}$ &{a {split-decomposition tree} rerooted at an
                     \emph{edge} connecting two star-nodes; in the general
                     case this can either be a center-center edge, or
                     an extremity-extremity edge; some classes, such as
                     ptolemaic graphs, may restrict this (and as before
                     the edge accounts for a subtree of each of the
                     nodes)}\\[0.4em]
    $\cls[S\rightarrow S]{T}$ &{a {split-decomposition tree} rerooted at a
                                \emph{directed edge}; similar to
                                $\cramped{\cls[S-S]{T}}$, except there is
                                a direction to the edge---and thus an order
                                to the star-nodes}\\[0.4em]
    \midrule
    $\cls{\bar{K}}$ &{a \emph{prohibitive} clique-node---used in the
                      grammar for ptolemaic graphs---entered through an
                      edge (and missing the corresponding subtree) that is
                      on a path that is connected to the center of a star;
                      this clique-node disallows outgoing connections to a
                      star-node's center, to avoid the formation of a
                      \emph{center-center path}, as stated by
                      Lemma~\ref{lem:split-characterization-c4}}\\[0.4em]
    $\cls[C]{Q}$ &{a \emph{``quadrilateral'' star-node}, as introduced in
                   the grammars for 4-cactus graphs of
                   Appendix~\ref{app:4-cactus}; this is one half of a group
                   of two star-nodes, each with two extremities, and linked
                   at their center as illustrated in
                   Figure~\ref{fig:split-subpattern-C4}; here we are
                   entering one such star-node from the center (or
                   equivalently the center of star-node is the subtree
                   that is missing), which means the parent
                   node/missing subtree is the other part of the two
                   star-node group}\\
    $\cls[X]{Q}$ &{a ``quadrilateral'' star-node, entered from an
                   extremity (or with a subtree rooted at an extremity
                   missing), which means that we must now connect the
                   center to a matching ``quadralateral'' star-node,
                   and the remaining extremity to something else}\\[0.4em]
    \bottomrule
  \end{tabularx}
  \caption{\label{tab:symbols}%
    The main symbols used to define the {split-decomposition tree} of the
    classes of graphs analyzed in this paper. Refer to
    Subsection~\ref{subs:split-grammars} for details on the terminology;
    and refer to Subsection~\ref{subs:dissymmetry} for details on the
    dissymmetry theorem, from which all the rerooted trees, denoted
    $\cramped{\cls[\omega]{T}}$, come from. (Note that for succinctness we
    have omitted the rerooted trees arising from applying the dissymmetry
    theorem to 4-cactus graphs in Appendix~\ref{app:4-cactus}.)}
\end{table*}


\section*{Acknowledgment}

We would like to thank Christophe Paul, for his help in understanding the
split-decomposition tree characterization of ptolemaic graphs.

All of the figures in this article were created in OmniGraffle~6~Pro. Some
were borrowed from Iriza~\cite{Iriza15}, and others were redone from
original figures by Gioan and Paul~\cite{GiPa12} (notably
Figures~\ref{fig:ex-split} and \ref{fig:reduced}).


\bibliographystyle{plain}
\bibliography{article2}{}

\appendix


\section{4-Cactus Graphs\label{app:4-cactus}}

While investigating block graphs and the related class of 3-cactus graphs,
we found an article by Harary and Uhlenbeck~\cite{HaUh53} in which they
investigate both 3-cactus graphs and 4-cactus graphs. This prompted us to
enumerate 4-cactus graphs.

This enumeration appears in appendix because it does not involve forbidden
induced subgraphs, and so is somewhat of a non-sequitur as far as the
point we would like to make in this paper.

The appeal of this enumeration is that it revisits a trick that is similar
to that used to enumerate ptolemaic graphs. In
Section~\ref{sec:ptolemaic}, we introduced two symbols to express
clique-nodes, $\cramped{\cls{K}}$ and $\cramped{\cls{\bar{K}}}$. These two
symbols were used to keep track of whether, in the rooted decomposition of
the split-decomposition tree, we were traveling down an alternated path
starting at the center of a star-node or not. This ``state'' information
was essential to prevent the formation of center-center paths that induce
$\cramped{C_4}$.

In this grammar for triangular cacti, we use a similar idea. The
quadrilaterals of these graphs translate to a very specific pattern in the
split-decomposition tree: two star-nodes of size 3, connected at their
center. We could well translate this in the grammar as a ``meta'' internal
node that is just those star-nodes combined. Instead, we define the two
symbols $\cramped{\cls[C]{Q}}$ and $\cramped{\cls[X]{Q}}$ to denote these
special star-nodes; and because they always come in pairs connected at
their centers, we know that if we encounter $\cramped{\cls[C]{Q}}$ we are
``inside'' the pattern, and if we encounter $\cramped{\cls[X]{Q}}$ we are
entering this pattern from the outside.

\begin{theorem}
  The class $\clsFCGrl$ of 4-cactus graphs rooted at a vertex is specified by
  \begin{align}
      \clsFCGrl      &= \clsAtom_{\mLeaf} \times (\cls[X]{Q}+\cls[C]{S})\\
      \cls[C]{Q}      &= \Set[=2]{\clsAtom+\cls[X]{S}}\\
      \cls[X]{Q}      &= \cls[C]{Q}\times(\clsAtom+\cls[X]{S})\\
      \cls[C]{S}      &= \Set[\geqslant2]{\cls[X]{Q}}\\
      \cls[X]{S}      &= \clsAtom\times\Set[\geqslant1]{\cls[X]{Q}}
    \end{align}
\end{theorem}


\begin{proof}[Sketch of proof.]
  We first note that 4-cacti's property of not having any induced
  clique-nodes of 3 or more vertices, by
  Lemma~\ref{lem:split-characterization-K4}, translates to the split-decomposition tree
  of 4-cacti having no clique-nodes. Therefore, the only internal nodes to
  consider are star-nodes.

  Furthermore, by Lemma~\ref{lem:split-characterization-c4}, every
  $\cramped{C_4}$ in a 4-cactus corresponds to a center-center path in the
  split-decomposition tree. Since we already ruled out the existence of clique-nodes,
  the only possible center-center paths in the split-decomposition tree of 4-cacti are
  two star-nodes adjacent via their centers, corresponding to an induced
  $\cramped{C_4}$ in the accessibility graph. Along this line, we
  distinguish between two types of star-nodes, with $\cls{Q}$ (for
  quadrilateral) representing star-nodes with their center adjacent to
  another star-node (another $\cls{Q}$) and $\cls{S}$ representing all
  other star-nodes, which we refer to as \emph{regular} star-nodes.

  We next observe that the centers of all regular star-nodes must be
  attached to leaves. They cannot be attached to extremities of other star
  nodes as that would allow for a star-join operation, and they cannot be
  attached to centers of other star-nodes since otherwise, they would be
  considered quadrilateral star-nodes instead of regular ones.

  Additionally, we note that extremities of regular star-nodes must be
  attached to quadrilateral star-nodes. First, these extremities cannot be
  attached to centers of other star-nodes to avoid star-join operations.
  Furthermore, as we already established that the centers of these regular
  star-node are attached to leaves, having their extremities adjacent to
  leaves or extremities of other star-nodes would induce bridges in the
  accessibility graph by Lemma~\ref{lem:split-characterization-bridge} (a
  pendant bridge in the former case, and an internal bridge in latter).
  However, every edge in a 4-cactus graph belongs to a $\cramped{C_4}$ and
  thus cannot be a bridge.

  Finally, we show that every quadrilateral star-node must have exactly
  two extremities. This is because there are no edges in an accessibility
  graph between the leaves at the ends of maximal alternated paths out of
  any star-node, as an alternated path between two such leaves would
  require using two interior edges from that star-node. Therefore, if two
  adjacent quadrilateral star-nodes respectively have $\cramped{x_1}$ and
  $\cramped{x_2}$ extremities, then the corresponding leaves in their
  subtrees would induce a $\cramped{K_{x_1, x_2}}$ in the accessibility
  graph. We have $\cramped{x_1}, \cramped{x_2}\geq2$ since the split-decomposition tree
  is assumed to be reduced. In a 4-cactus graph, the only allowed complete
  bipartite induced subgraph, where each side of the bipartition has size
  at least 2, is a $\cramped{K_{2,2}}$, i.e. $\cramped{C_4}$, implying
  that every quadrilateral node must have exactly two extremities.
\end{proof}

\begin{theorem}
  The class $\clsFCG$ of unrooted 4-cactus graphs is specified by
  \begin{align}
    \clsFCG        &= \cls[Q]{T} +\cls[S]{T} + \cls[Q-Q]{T}%
                       - \cls[Q\rightarrow Q]{T} - \cls[Q-S]{T} \\[0.4em]
    \cls[Q]{T}       &= \cls[C]{Q}\times\cls[C]{Q}\\
    \cls[S]{T}       &= \clsAtom\times\cls[C]{S} \\
    \cls[Q-Q]{T}   &= \Set[2]{\cls[C]{Q}}\\
    \cls[Q\rightarrow Q]{T}
                   &= \cls[C]{Q}\times\cls[C]{Q}\\
    \cls[Q-S]{T}   &= \cls[X]{Q}\times\cls[X]{S}\\[0.4em]
    \cls[C]{Q}      &= \Set[=2]{\clsAtom+\cls[X]{S}}\\
    \cls[X]{Q}      &= \cls[C]{Q}\times(\clsAtom+\cls[X]{S})\\
    \cls[C]{S}      &= \Set[\geqslant2]{\cls[X]{Q}}\\
   \cls[X]{S}      &= \clsAtom\times\Set[\geqslant1]{\cls[X]{Q}}
  \end{align}
\end{theorem}

\noindent The first few terms of the enumeration of unlabeled, unrooted
4-cactus graphs are available in the Table~\ref{tab:enum-cacti} at the end
of this paper.

\section{Proof of Lemmas~\ref{lem:alternated-paths},
  \ref{lem:alternated-paths-disjoint},
  and~\ref{lem:split-induced-clique}\label{app:proof-tree-lemmas}}

We here restate and provide the full proof of two straight-forward lemmas,
which 

\tmpsetcounter{lemma}{3}
\begin{lemma}
  Let $G$ be a totally decomposable graph with the reduced clique-star
  split-decomposition tree $T$, any maximal\footnote{A maximal alternated path is one
    that cannot be extended to include more edges while remaining
    alternated.} alternated path starting from any node in $V(T)$ ends in
  a leaf.
\end{lemma} 

\begin{proof}
  Let $P$ be a maximal alternated path of length $\ell$ (edges) originating
  from a node $u\in V(T)$, and suppose $P$ does not end in a leaf. Let $v$
  be the internal node that $P$ ends in, and let $x\in V(\cramped{G_v})$
  be the marker vertex attached to the edge in $P$ that enters $v$. Note
  that since $P$ is an alternated path, it can include at most one edge
  from $E(\cramped{G_v})$, so the part of $P$ going from $u$ to $x$ has
  length at least $\ell-1$ edges. Depending on the structure of
  $\cramped{G_v}$, we have the following cases:
  \begin{itemize}[noitemsep, nosep]
  \item \emph{$\cramped{G_v}$ is a clique-node}.\ \ %
    Since $T$ is reduced, $v$ has degree at least 3. Therefore, $x$ has at
    least two neighbors in $V(\cramped{G_v})$. Let $y$ be one of these
    neighbors, and let $\cramped{\rho_y}$ connect $y$ to $z$ ($z$ is
    either a leaf or a marker vertex in a different internal node than
    $v$). We can construct a new path $P'$ of length $\ell+1$ from $P$ by
    cutting off the part of $P$ that follows $x$ and adding $(x,y)$ and
    $(y,z)$ to the end of $P$, contradicting the maximality of $P$.
  \item \emph{$\cramped{G_v}$ is a star-node and $x$ is its center}.\ \ %
    Similarly to the previous case, $\cramped{G_v}$ has degree at least 3
    and thus at least two extremities. Let $y\in V(\cramped{G_v})$ be one
    of these extremities and repeat the same argument in the previous
    case.
  \item \emph{$\cramped{G_v}$ is a star-node and $x$ is one of its
      extremities}.\ \ %
    Let $y\in V(\cramped{G_v})$ be the center of the star-node $v$, and
    the same argument given for the previous two cases applies.
  \end{itemize}
  
\end{proof}

\begin{lemma}
  Let $G$ be a totally decomposable graph with the reduced clique-star
  split-decomposition tree $T$ and let $u\in V(T)$ be an internal node. Any two maximal
  alternated paths $P$ and $Q$ that start at distinct marker vertices of
  $u$ but contain no interior edges from $\cramped{G_u}$ end at distinct
  leaves.
\end{lemma}
\begin{proof}
  Let $P$ and $Q$ be two such maximal alternated paths, starting at marker
  vertices $p,q\in V(\cramped{G_u})$ respectively. We will show that $P$
  and $Q$ are disjoint, and the result follows.

  Note that $P$ and $Q$ are indeed disjoint in $\cramped{G_u}$, since they
  begin at different marker vertices and leave $u$ immediately from there.
  Suppose $P$ and $Q$ are not disjoint, and let $v\in V(T)$ be the first
  such common node encountered when tracing $P$ and $Q$ out of $u$. Since
  $P$ and $Q$ are disjoint before reaching $v$, the part of $P$ and $Q$
  between $u$ and $v$ form a simple cycle in $T$, which cannot happen in a
  tree. 
  Therefore, no such common node $v$ can exist, implying $P\cap Q = \{\}$.

\end{proof}

\begin{lemma}
  Let $G$ be a totally decomposable graph with the reduced clique-star
  split-decomposition tree $T$. If $T$ has a clique-node of degree $n$,
  then $G$ has a corresponding induced clique on (at least) $n$ vertices.
\end{lemma}

\begin{proof}
  Let $u\in V(T)$ be a clique-node of degree $n$. For every marker vertex
  $\cramped{v_i}\in V(\cramped{G_u})$, $i = 1\dots n$, fix some maximal
  alternated path $\cramped{P_{v_i}}$ that starts at the marker vertex
  $\cramped{v_i}$, uses no interior edges of $\cramped{G_u}$, and ends at
  leaf $\cramped{a_i}$. Note that the $\cramped{a_i}$ are all distinct, by
  Lemma~\ref{lem:alternated-paths-disjoint}, and pairwise adjacent, since
  every pair of leaves $\cramped{a_i}, \cramped{a_j}$ are connected in $T$
  via the alternated path consisting of $\cramped{P_{v_i}}$,
  $\cramped{P_{v_j}}$, and the edge
  $(\cramped{v_i},\cramped{v_j})\in E(\cramped{G_u})$, thus inducing a
  clique of size (at least\footnote{Note that this clique of size $n$
    might be part of a larger clique, as illustrated in
    Figure~\ref{fig:clique-node-not-clique}.}) $n$ in $G$.
  \vspace{4em}
  
  \begin{figure}[h!]
    \centering
    \includegraphics[scale=0.5]{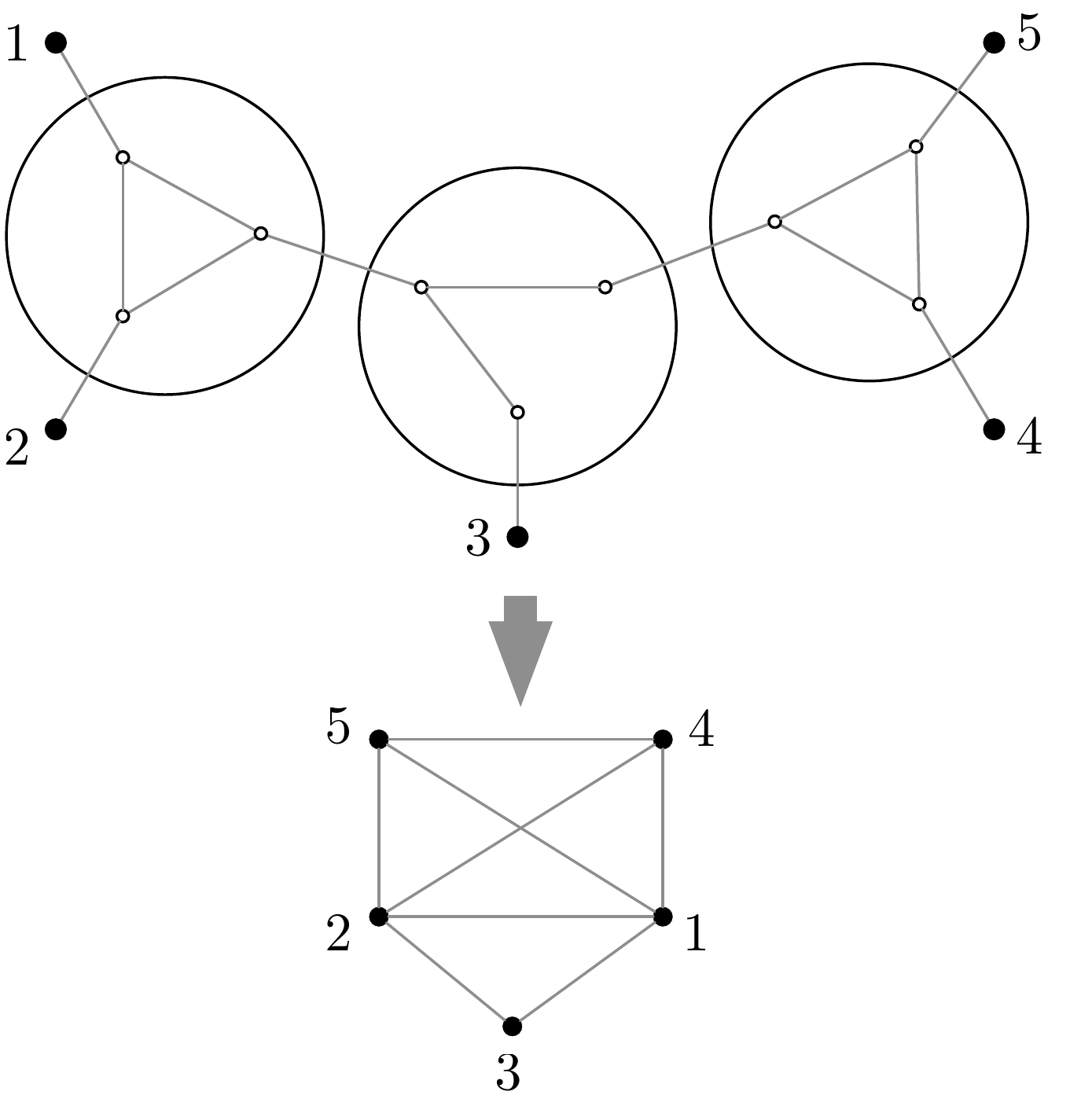}
    \caption{\label{fig:clique-node-not-clique} %
      While the two clique-nodes of size 3 guarantee the presence of two
      corresponding induced cliques (one involving vertices 1 and 2, the
      other involving vertices 4 and 5), they do not allow us to rule out
      the existence of larger clique. This illustrates that, unlike many
      of our lemmas, the property presented in
      Lemma~\ref{lem:split-induced-clique} is not bijective, and only
      works in one direction.}
  \end{figure}
\end{proof}

\tmprestorecounter{lemma}


\begin{table*}
  \centering
  \def\arraystretch{1.5}
  \begin{tabularx}{.9\textwidth}{c c c c W}
    \toprule
    \textbf{Graph Class}&\textbf{Rooted}&\textbf{Labeled}&\textbf{EIS}&\textbf{Enumeration}\\
    
    \midrule

    Block graphs &{\checkmark} &{\checkmark} &{\textbf{A035051}} &{1, 2, 12,
      116, 1555, 26682, 558215, 13781448, 392209380, 12641850510,
      455198725025, 18109373455164, 788854833679549, \dots}\\

    Block graphs &{} &{\checkmark} &{\textbf{A030019}} &{1, 1, 4, 29, 311,
      4447, 79745, 1722681, 43578820, 1264185051, 41381702275, 1509114454597,
      60681141052273,\dots}\\

    Block graphs &{\checkmark} &{} &{\textbf{A007563}} &{1, 1, 3, 8, 25, 77,
      258, 871, 3049, 10834, 39207, 143609, 532193, 1990163, 7503471,
      28486071, 108809503, 417862340,\dots}\\

    Block graphs &{} &{} &{\textbf{A035053}} &{1, 1, 2, 4, 9, 22, 59, 165,
      496, 1540, 4960, 16390, 55408, 190572, 665699, 2354932, 8424025,
      30424768, 110823984,\dots}\\
      
    \midrule

    Ptolemaic graphs &{\checkmark} &{\checkmark} &{} &{1, 2, 12, 140, 2405,
      54252, 1512539, 50168456, 1928240622, 84240029730, 4121792058791,
      223248397559376, \dots}\\

    Ptolemaic graphs &{} &{\checkmark} &{} &{1, 1, 4, 35, 481, 9042, 216077,
      6271057, 214248958, 8424002973, 374708368981, 18604033129948,
      1019915376831963, \dots}\\

    Ptolemaic graphs &{\checkmark} &{} &{}  &{1, 1, 3, 10, 40, 168, 764,
      3589, 17460, 86858, 440507, 2267491, 11819232, 62250491, 330794053,
      1771283115, 9547905381, \dots}\\
 
     Ptolemaic graphs &{} &{} &{} &{1, 1, 2, 5, 14, 47, 170, 676, 2834,
       12471, 56675, 264906,1264851, 6150187, 30357300, 151798497,
       767573729, 3919462385, \dots}\\
    
    \bottomrule
  \end{tabularx}
  \caption{\label{tab:enum-block-ptol}%
    The first few terms of the enumerations of ptolemaic and block graphs.}
\end{table*}

\begin{table*}
  \centering
  \def\arraystretch{1.5}
  \begin{tabularx}{.9\textwidth}{c c c c W}
    \toprule
    \textbf{Graph Class}&\textbf{Rooted}&\textbf{Labeled}&\textbf{EIS}&\textbf{Enumeration}\\
    \midrule

    2,3-cactus graphs &{\checkmark} &{\checkmark} &{\textbf{A091481}} &{1,
      2, 12, 112, 1450, 23976, 482944, 11472896, 314061948, 9734500000,
      336998573296,12888244482048, \dots}\\

    2,3-cactus graphs &{} &{\checkmark} &{\textbf{A091485}} &{1, 1, 4,
      28, 290, 3996, 68992, 1434112, 34895772, 973450000, 30636233936,
      1074020373504, 41510792057176, \dots}\\

    2,3-cactus graphs &{\checkmark} &{} &{\textbf{A091486}} &{1, 1, 3, 7,
      21, 60, 190, 600, 1977, 6589, 22408, 77050, 268178, 941599, 3333585,
      11882427, 42615480,153653039, \dots}\\

    2,3-cactus graphs &{} &{} &{\textbf{A091487}} &{1, 1, 2, 3, 7, 16,
      41, 106, 304, 880, 2674, 8284, 26347, 85076, 279324, 928043,
      3118915, 10580145, 36199094, 124774041, \dots}\\ 

    \midrule

    3-cactus graphs &{\checkmark} &{\checkmark} &{\textbf{A034940}} &{0,
      0, 3, 0, 75, 0, 5145, 0, 688905, 0, 152193195, 0, 50174679555, 0,
      23089081640625, 0, \dots}\\
    
    3-cactus graphs &{} &{\checkmark} &{\textbf{A034941}} &{0, 0, 1, 0,
      15, 0, 735, 0, 76545, 0, 13835745, 0, 3859590735, 0, 1539272109375,
      0, 831766748637825, 0, \dots}\\

    3-cactus graphs &{\checkmark} &{} &{\textbf{A003080}} &{0, 0, 1, 0,
      2, 0, 5, 0, 13, 0, 37, 0, 111, 0, 345, 0, 1105, 0, 3624, 0, 12099,
      0, 41000, 0, 140647, 0, 487440, 0, 1704115, 0, \dots}\\
    
    3-cactus graphs &{} &{} &{\textbf{A003081}} &{0, 0, 1, 0, 1, 0, 2,
      0, 4, 0, 8, 0, 19, 0, 48, 0, 126, 0, 355, 0, 1037, 0, 3124, 0,
      9676, 0, 30604, 0, 98473, 0, 321572, 0, 1063146, 0, \dots}\\   

    \midrule

    4-cactus graphs &{\checkmark} &{\checkmark} &{} &{0, 0, 0, 12, 0, 0,
      4410, 0, 0, 7560000, 0, 0, 35626991400, 0, 0, 357082280755200, 0, 0,
      6536573599765809600, 0, 0, \dots}\\

    4-cactus graphs &{} &{\checkmark} &{} &{0, 0, 0, 3, 0, 0, 630, 0, 0,
      756000, 0, 0, 2740537800, 0, 0, 22317642547200, 0, 0,
      344030189461358400, 0, 0, \dots}\\

    4-cactus graphs &{\checkmark} &{} &{} &{0, 0, 0, 1, 0, 0, 3, 0, 0,
      11, 0, 0, 46, 0, 0, 208, 0, 0, 1002, 0, 0, 5012, 0, 0, 25863, 0,
      0, 136519, 0, 0, 733902, 0, 0, \dots}\\

    4-cactus graphs &{} &{} &{} &{0, 0, 0, 1, 0, 0, 1, 0, 0, 3, 0, 0, 7,
      0, 0, 25, 0, 0, 88, 0, 0, 366, 0, 0, 1583, 0, 0, 7336, 0, 0,
      34982, 0, 0, 172384, 0, 0, \dots}\\ 
    
    \bottomrule
  \end{tabularx}
  \caption{\label{tab:enum-cacti}%
    The first few terms of the enumerations of some subclasses of cactus
    graphs studied in this paper. The zero terms in the enumeration of
    3-cacti and 4-cacti are due to the fact that the number of vertices in
    a 3-cactus graph must be odd, while the number of vertices in a
    4-cactus graph must have a remainder of 1 modulus 3. Note that the EIS
    sequence for 3-cacti lists the enumeration for graphs of odd size
    only, thus omitting the zero terms.}
\end{table*}

\end{document}